\documentclass[a4paper,11pt, reqno]{amsart}
\usepackage[T1]{fontenc}
\usepackage[english]{babel}

\usepackage{amsmath}
\usepackage{amssymb}
\usepackage{amsthm}
\usepackage[foot]{amsaddr}
\usepackage{bbm}
\usepackage{mathrsfs}
\usepackage{epsfig}
\usepackage{enumerate}
\usepackage{mathtools}
\usepackage{graphicx}
\usepackage[usenames,dvipsnames]{xcolor}
\usepackage[ pdftitle={Brownian windings, Stochastic Green's formula and inhomogeneous magnetic impurities},pdfauthor={Isao Sauzedde} colorlinks=true,linkcolor=NavyBlue,urlcolor=RoyalBlue,citecolor=PineGreen,bookmarks=true,bookmarksopen=true,bookmarksopenlevel=2,unicode=true,linktocpage]
{hyperref}
\usepackage[margin=1in]{geometry}

\newtheorem*{lemma*}{Lemma}
\newtheorem{theoremi}{Theorem}
\newtheorem{corollaryi}[theoremi]{Corollary}
\newtheorem{remarki}[theoremi]{Remark}
\newtheorem{theorem}{Theorem}[section]
\newtheorem{lemma}[theorem]{Lemma}

\newtheorem{corollary}[theorem]{Corollary}

\renewcommand*{\d}{{\mathop{}\!\mathrm{d}}}

\def\Xint#1{\mathchoice
{\XXint\displaystyle\textstyle{#1}}%
{\XXint\textstyle\scriptstyle{#1}}%
{\XXint\scriptstyle\scriptscriptstyle{#1}}%
{\XXint\scriptscriptstyle\scriptscriptstyle{#1}}%
\!\int}
\def\XXint#1#2#3{{\setbox0=\hbox{$#1{#2#3}{\int}$}
\vcenter{\hbox{$#2#3$}}\kern-.5\wd0}}

\def\R{{\mathbb R}}

\def\n{{\mathbf n}}

\DeclareMathOperator{\Range}{Range}

\DeclareMathOperator{\sinc}{sinc}

\DeclareMathOperator{\sgn}{sgn}

\def\Xint#1{\mathchoice
   {\XXint\displaystyle\textstyle{#1}}%
   {\XXint\textstyle\scriptstyle{#1}}%
   {\XXint\scriptstyle\scriptscriptstyle{#1}}%
   {\XXint\scriptscriptstyle\scriptscriptstyle{#1}}%
   \!\int}
\def\XXint#1#2#3{{\setbox0=\hbox{$#1{#2#3}{\int}$}
     \vcenter{\hbox{$#2#3$}}\kern-.5\wd0}}

\def\dashint{\Xint-}
\def\cut#1#2{\operatorname{cut}_#1(#2)}

\title[Brownian windings, Stochastic Green's formula and inhomo. magnetic impurities]{Brownian windings, Stochastic Green's formula and inhomogeneous magnetic impurities}
\author{Isao Sauzedde }
\address{ENS de Lyon}
\email{isao.sauzedde@ens-lyon.fr}


\begin{document}
\begin{abstract}
We give a general Green formula for the planar Brownian motion, which we apply to study the Aharonov--Bohm effect induced by Poisson distributed magnetic impurities on a Brownian electron in the presence of an inhomogeneous magnetic field.
\end{abstract}
\maketitle

\setcounter{tocdepth}{1}
\tableofcontents

\section{Introduction}
\subsection{Stochastic Green's formula}
For a smooth loop $X=(X^1,X^2):[0,T]\to \mathbb{R}^2$ and a point $z$ outside the range of $X$, let $\n_X(z)\in \mathbb{Z}$ be the winding index of $X$ around $z$. For any smooth differential $1$-form $\eta=\eta_1\d x^1+\eta_2\d x^2$, the Green formula states\footnote{Green formula is usually stated for the particular case of loops with no self-intersections and run through counterclockwise. Then, the winding function $\n_X$ is nothing but the indicator function of the bounded connected component $D$ delimited by $X$, so the right-hand side of~\eqref{eq:green} reduces to $\int_D \d \eta$.}  that
\begin{equation}
\label{eq:green}
\int_X \eta = \int_{\mathbb{R}^2} \n_X \d \eta,
\end{equation}
where $\d \eta=(\partial_1 \eta_2(z)-\partial_2 \eta_1(z) ) \d z$ is the exterior derivative of $\eta$.
In other words, for two smooth functions $\eta_1,\eta_2:\mathbb{R}^2\to \mathbb{R}$,
\[
\int_0^T \eta_1(X_t) \d X^1_t+\int_0^T \eta_2(X_t) \d X^2_t = \int_{\mathbb{R}^2} \n_X(z) (\partial_1 \eta_2(z)-\partial_2 \eta_1(z) ) \d z.
\]
Consider now the case of a Brownian loop.
A natural candidate for the left-hand side is given by the Stratonovich integrale of $\eta$ along $X$. For the right-hand side however, the index function $\n_X$ blows up slightly too fast in the vicinity of $X$, which prevents it from being integrable \cite{Werner3}. The goal of this paper is to prove that this integral admits a sort of principal value, for which the Green's formula holds true.

The method we use to define this principal value is we cut-off the extreme values of $\n_X(z)$, in a way which is symmetric (i.e. the positive and negative cut-off are identical), and let the cut-off goes to infinity \emph{outside from the integral}. In this way, the can take advantage of some cancellation between the positive and negative contributions to the integral. Roughly speaking, there is about as much points $z$ for which $\n_X(z)$ is larger than a given $k$ than points $z$ for which $\n_X(z)$ is smaller than $-k$: as $k$ goes to infinity, the difference between the size of these two sets is much smaller than the size of any one of these two sets. In fact, this is true not only when the size is computed with respect to the Lebesgue measure, but with respect to \emph{any smooth enough measure}, in particular for the measure $(\partial_1 \eta_2(z)-\partial_2 \eta_1(z) ) \d z$.

We shall consider two different ways to cut-off extreme values:
for $x\in \mathbb{R}$ and $N>0$, we define $\cut{N}{x}\coloneqq \max( \min(x,N ),-N)$, but the results also hold when $\cut{N}{x}$ is replaced with  $x \mathbbm{1}_{|x|\leq N}$. For $\beta\in(0,1)$, let
$\mathcal{C}_2^{\beta}(\mathbb{R}^2)$ be the set of bounded function which are $\beta$-H\"older continuous and square-integrable, and $\mathcal{C}^{1+\beta}_2(\mathbb{R}^2)$ be the set of $\mathcal{C}^{1}$ functions $g$ such that $\partial_1 g,\partial_2 g\in \mathcal{C}_2^{\beta}(\mathbb{R}^2)$.
Let finally $\mathcal{C}^{1+\beta}_2(\mathbb{R}^2,\mathbb{R}^2)$ be the set of differential $1$-forms $\eta=\eta_1\d x+\eta_2\d y$ such that $\eta_1,\eta_2\in\mathcal{C}^{1+\beta}_2(\mathbb{R}^2)$.

\begin{theoremi}
  \label{th:green}
  Let $X:[0,T]\to \mathbb{R}^2$ be a Brownian motion, and let $\n_X$ be the winding function associated with the loop obtained by concatenation of $X$ with the straight line segment $[X_T,X_0]$ between its endpoints.
  Then, almost surely, for all $f\in \mathcal{C}_2^{\beta}(\mathbb{R}^2)$ for some $\beta>0$,
  \[
  \int_{\mathbb{R}^2} \cut{N}{\n_X(z)}  f(z) \d z\]
  converges as $N\to \infty$. Define $  \dashint_{\mathbb{R}^2} \n_X  f \d \lambda$ as the almost sure limit.

  Furthermore, if $\eta\in \mathcal{C}^{1+\beta}_2(\mathbb{R}^2,\mathbb{R}^2)$ is such that $f=\partial_1\eta_2-\partial_2\eta_1$, then, almost surely,
  \begin{equation}
  \label{eq:thgreensto}
  \int_0^T \eta\circ \d X + \int_{[X_T,X_0]} \eta=  \dashint_{\mathbb{R}^2} \n_X  f \d \lambda ,
  \end{equation}
  where the stochastic integral in the left-hand side is to be understood in the sense of Stratonovich.
\end{theoremi}
\begin{corollaryi}
  \label{coro:green}
  For all $x$ and $y$ in $\mathbb{R}^2$, the same result holds if the planar Brownian motion is replaced with a planar  Brownian loop or a planar Brownian bridge between distinct points.
\end{corollaryi}

This generalises a result the present author obtained in~\cite{LAWA} in the special case when $\eta(x)=x^1\d x^2$: in this case, the right-hand side in~\eqref{eq:thgreensto} is the Lévy area, and $\partial_1\eta_2-\partial_2\eta_1$ is constant, which simplify much the problem. The proof here uses the tools from~\cite{LAWA}, but also some estimations the author obtained later in \cite{BWE}.

Other forms of stochastic Green's formula, using different regularisation methods, were obtained in \cite{Werner2}. The key advantages of our method is that we can obtain almost sure convergence (thus $\dashint_{\mathbb{R}^2} \n_X  f\d \lambda$ is defined deterministically), and  the function we integrate when the cut-off is still present only depends on the value of the winding function \emph{at this exact point} rather than on its vicinity\footnote{In particular, the definition of $\dashint_{\mathbb{R}^2} \n_X  f \d \lambda$ is invariant under diffeomorphism, thus intrinsic to the data of $\eta$ and $X$ in the smooth surface defined by $\mathbb{R}^2$, rather than dependent on its specific identification with $\mathbb{R}^2$.}. This will allow us to prove a $1$-stable central limit theorem for a Monte--Carlo type of approximation of~\eqref{eq:thgreensto}, where the integral is replaced with a sum over a Poisson point process with large intensity on the plane, as we now explain.

\subsection{Magnetic impurities}
In the second part of the paper, we consider the weighted and averaged winding of a Brownian motion $X$ around a Poisson distributed set of points $\mathcal{P}$, which is partly motivated by mesoscopic models for electrons in disordered systems. Here, the Poisson point process models magnetic impurities inside a 2 dimensional medium while the Brownian path models a quasiclassical electron scattering inside this medium. The impurities interact with the electron via the Aharonov--Bohm effect and produce the phase shift  $\exp( i\alpha  \sum_{z\in \mathcal{P}} f(z) \n_X(z)  ) $. In \cite{Desbois2}, the failure of the law of large numbers in this situation was already recognised and the correction due to the Cauchy tail of Brownian windings is correctly computed, in the situation the weight function $f$ is constant and the intensity of the point process is large (and inversely proportional to $\alpha$), but translation-invariant. See \cite{physmes,localization,Desbois, inhom3} and references therein for the physics literature on the topic. We consider here the general situation where these two assumptions are dropped, with the effect that the Cauchy parameters are now functions of $X$ : in the previous special case, the relevant expectations with respect to $\mathcal{P}$, i.e. conditional on $X$, are asymptotically deterministic. In the general case, this is no longer the case and we express them explicitly in term of $X$.

\begin{theoremi}
  \label{th:magne}
  Let $f,g\in \mathcal{C}^{\beta}_2(\mathbb{R}^2 )$ for some $\beta>0$, with $g\geq 0$. For $\rho>0$, let $\mathcal{P}$ be Poisson process on $\mathbb{R}^2$ with intensity $\rho g\d \lambda$, where $\d \lambda$ is the Lebesgue measure on $\mathbb{R}^2$, and let $X$ be either a Brownian motion or a Brownian bridge with duration $1$, independent from $\mathcal{P}$. Then, $X$-almost surely, for all $\alpha\in \mathbb{R}$,
  \[
  \lim_{\rho\to \infty} \mathbb{E}^{\mathcal{P}}\big[ \exp( i \frac{\alpha}{\rho}  \sum_{z\in \mathcal{P}} \n_X(z)f(z)  )\big]=
  \exp\Big( i\alpha \dashint \n_Xf g \d\lambda - \frac{|\alpha|}{2} \int_0^1 |f(X_t)|g(X_t)\d t  \Big)
  \]
  where $\mathbb{E}^{\mathcal{P}}=\mathbb{E}[\ \cdot \ |X]$ is the expectation over $\mathcal{P}$.
\end{theoremi}
By Lévy's continuity theorem and using~\eqref{eq:thgreensto}, this is equivalent to the following $1$-stable central limit theorem, which replaces the usual law of large numbers which would hold if the function $\n_X$ was integrable:
\begin{theoremi}
$X$-almost surely, in distribution over $\mathcal{P}$,
\[\frac{1}{\rho}  \sum_{z\in \mathcal{P}} \n_X(z) f(z)\underset{\rho\to \infty}\longrightarrow
\mathscr{C}\Big(\int_0^T \eta\circ \d X + \int_{[X_T,X_0]} \eta\ , \ \frac{1}{2}\int_0^1 |f(X_t)|g(X_t)\d t \Big),\]
where $\mathscr{C}(p,\sigma)$ is the Cauchy distribution with position parameter $p$ and scale parameter $\sigma$, and $\eta$ is a differential form such that $\d \eta(x) =f(x)g(x)\d x$.
\end{theoremi}
We then easily deduce the  following stronger result. 
\begin{corollaryi}
  \label{coro:magne}
  Let $g\in \mathcal{C}^{\beta}_2(\mathbb{R}^2 )$, with $g\geq 0$. For $\rho>0$, let $\mathcal{P}$ be Poisson process on $\mathbb{R}^2$ with intensity $\rho g(z)\d z$, and $X$ be either a Brownian motion or a Brownian bridge with duration $1$, independent from $\mathcal{P}$. Let also $\Gamma:[0,1]\to \mathbb{R}$
  be a standard Cauchy process. Then, for all $(f_1,\dots, f_n)\in \mathcal{C}^{\beta}_2(\mathbb{R}^2)$, $X$-almost surely, the $n$-uple
  \[ \Big( \frac{1}{\rho}  \sum_{z\in \mathcal{P}} f_1(z)\n_X(z) , \dots, \frac{1}{\rho}  \sum_{z\in \mathcal{P}} f_n(z)\n_X(z) \Big)\]
  converges in distribution toward $
  (\xi(f_1),\dots, \xi(f_n) )$
  where
  \[\xi(f)\coloneqq \dashint \n_X f g\d \lambda+\frac{1}{2} \int_0^1 f(X_t)g(X_t) \d \Gamma_t.\]
\end{corollaryi}
This has a nice rough-path theoretic interpretation which goes roughly as follows: one can think of the environment as macroscopically charged with the density of charge $g$, in that for example, along any smooth loop $\gamma$, the total flux $ \frac{1}{\rho}  \sum_{z\in \mathcal{P}} f_1(z)\n_\gamma(z)$ will converge toward $\int_{\mathbb{R}^2} f_1 \n_\gamma g \d \lambda$. However, for a loop as rough as a Brownian motion, the microscopic details do matter and the Poissonisation has the effect, event in the large $\rho$ limit, of shifting the canonical Stratonovich rough path extension of $X$ whose antisymmetric part is $(\mathbb{A}_{s,t})_{s<t}$ to
$( \mathbb{A}_{s,t}+\Gamma_t-\Gamma_s)_{s<t}$ in the case $g=1$, or in the general case to
$( \mathbb{A}_{s,t}+\tilde{\Gamma}_{L_t}-\tilde{\Gamma}_{L_s})_{s<t}$, with $\tilde{\Gamma}$ a copy of $\Gamma$ and $L_t=\int_0^t g(X_u)\d u$.
As far as the author is aware, this is the first known natural occurrence of a non-standard rough path extension which is related to the microscopic details of the environment (in the case e.g. of the physical Brownian motion in magnetic field \cite{frizgassiatlyons}, the microscopic details of the \emph{path} play the central role).

\begin{remarki}
  Since all the results hold $X$-almost surely, and since the functions $f,g$ are continuous, the assumptions that they are square-integrable can easily be lifted. Some of the intermediate quantitative results depends upon their $L^\infty(\mathbb{R}^2)$ norms; working under the $L^2(\mathbb{R}^2)$ assumption makes is easier to obtain estimations in $L^p(\Omega)$ during intermediate steps.
\end{remarki}
\begin{remarki}
  Given $f,g\in \mathcal{C}^\beta_2(\mathbb{R}^2)$, there always exists a differential $1$-form $\eta$ with regularity $\mathcal{C}^{1+\beta} $ such that $\partial_1\eta_2-\partial_2\eta_1=fg$, so that $\dashint \n_X f g \d \lambda$ can always be written as a stochastic integral via~\eqref{eq:thgreensto}.
\end{remarki}

This paper is built in the continuity of \cite{LAWA} and \cite{BWE}, and we rely heavily on the results these papers contain.

\section{Notations}
\subsection{Differential forms and integrals}
For  $f\in \mathcal{C}^\beta(\mathbb{R}^2)$, define the semi-norm
\[
|f|_{\mathcal{C}^\beta}\coloneqq \sup_{\substack{x,y\in \mathbb{R}^2\\ x\neq y}} \frac{ f(x)-f(y)  }{|x-y|^\beta}.\]
The space $\mathcal{C}_2^\beta(\mathbb{R}^2)=\mathcal{C}^\beta(\mathbb{R}^2)\cap L^2(\mathbb{R}^2)\subset \mathcal{C}^\beta(\mathbb{R}^2)\cap L^\infty(\mathbb{R}^2)$ is endowed with the norm
\[ \|f\|_{\mathcal{C}^\beta_\infty}=\|f\|_\infty+|f|_{\mathcal{C}^\beta}.
\]
Given a curve $X:[0,T]\to \mathbb{R}^2$, we write
\[\int_X \eta \coloneqq \int_0^T \eta_1(X_t) \d X^1_t+\int_0^T \eta_2(X_t) \d X^2_t,\] where these integrals are to be understood either as classical integrals or as Stratonovich integrals, depending on the regularity of $X$. No It\^o integral will be involved in this paper, and all the stochastic integrals are to be understood in the sense of Stratonovich.


For $\eta\in \mathcal{C}^{1+\beta}(\mathbb{R}^2,\mathbb{R}^2)$, we identify the $2$-form $\d \eta=(\partial_1\eta_2-\partial_2\eta_1)\d x^1\wedge \d x^2$ with the signed measure $(\partial_1\eta_2-\partial_2\eta_1)\d x$, where $\d x$
is the Lebesgue measure on $\mathbb{R}^2$.

For a bounded set $\mathcal{D}\subset \mathbb{R}^2$ and $f\in L^1_{loc}(\mathbb{R}^2)$, we use the unconventional notation
\[f(\mathcal{D})\coloneqq\int_{\mathcal{D}} f(z)\d z,\]
and $|\mathcal{D}|$ for the Lebesgue measure of $\mathcal{D}$. Many subsets of $\mathbb{R}^2$ are written with curly letters, e.g. $\mathcal{D}_X$, and we then use the same notation but with a straight letter, e.g. $D_X$, for the Lebesgue measure of these sets.

%
%

 \subsection{Winding}

Given a curve $X$ on $\mathbb{R}^2$, that is a continuous function from $[0,T]$ to $\mathbb{R}^2$ for some $T>0$, we write $\bar{X}$ for the loop obtained by concatenation of $X$ with a straight line segment from $X_T$ to $X_0$. Although the parameterisation of this line segment does not matter in the following, we will assume it is parameterized by $[T,T+1]$ at constant speed, unless $X$ is already a loop (that is, a curve with $X_T=X_0$). In this case when $X$ is a loop, we set instead $\bar{X}\coloneqq X$.

Given a curve $X$ and a point $z$ outside the range of $\bar{X}$, we write $\n_X(z)$ for the winding number of $\bar{X}$ around $z$.

For a relative integer $k$, we define
\[ \mathcal{A}_k^X=\{z\in \mathbb{R}^2\setminus \Range(\bar{X}): \n_X(z)=k\}.\]
 For $n>0$, we also define
\[ \mathcal{D}_n^X=\{ z\in \mathbb{R}^2\setminus \Range(\bar{X}): \n_X(z)\geq n\}= \bigsqcup_{k\geq n} \mathcal{A}_k^X,\]
and
\[ \mathcal{D}_{-n}^X=\{ z\in \mathbb{R}^2\setminus \Range(\bar{X}): \n_X(z)\leq - n\}= \bigsqcup_{k \leq -n } \mathcal{A}_k^X.\]
We also write $A_k^X$ (resp. $D_k^X$) for the Lebesgue measure of $\mathcal{A}_k^X$ (resp. $\mathcal{D}_k^X$), and we remove the superscript $X$ when we consider one single curve.

Recall the definition
\[\cut{N}{z}\coloneqq
\left\{\begin{array}{cc}
-N &\mbox{if } z\leq -N,\\
z &\mbox{if } -N\leq z\leq N,\\
N &\mbox{if } z\geq N.\\
\end{array}
\right.
\]
Whenever it exists, we write $\dashint_{\mathbb{R}^2}  n_X f\d \lambda $ for the almost sure limit
\[ \lim_{N\to \infty} \int_{\mathbb{R}^2} \cut{N}{n_X(z)} f(z) \d z.\]
%
%

\subsection{Cauchy variables}

The Cauchy distribution $\mathcal{C}(p,\sigma)$ with position parameter $p$ and scale parameter $\sigma>0$ is the probability distribution on $\mathbb{R}$ which has a density with respect to the Lebesgue measure given at $x$ by
\[
\frac{1}{\pi \sigma} \frac{\sigma^2}{\sigma^2+(x-p)^2 }.
\]
In order to unify some results, we also include the degenerate case $\sigma=0$: we write $\mathcal{C}(p,0)$ for the Dirac measure at $p$.

Following \cite[Definition 5.2]{JohnsonSamworth}\footnote{Again, as opposed to \cite{JohnsonSamworth}, we include the trivial case $\sigma=0$ in our definition.}, we will say that a random variable $Z$ on $\mathbb{R}$ lies in the \emph{strong domain of attraction} of a Cauchy distribution if there exists $\sigma\geq 0, \delta>0$ such that
\[ \mathbb{P}(Z\geq x ) \underset{x\to +\infty}= \frac{\sigma}{\pi x}+o(x^{-(1+\delta)}), \qquad  \mathbb{P}(Z\leq -x )  \underset{x\to +\infty}= \frac{\sigma}{\pi x}+o(x^{-(1+\delta)}).\]
It then follows from Lemma 5.1 and Theorem 1.2 in \cite{JohnsonSamworth} that $Z$ follows a central limit theorem: if $(Z_i)_{i\in \mathbb{N}}$ are i.i.d. copies of $Z$, then there exists a unique $p$ such that
\[\frac{1}{N}\sum_{i=1}^N Z_i \Longrightarrow Y\sim \mathcal{C}(p,\sigma).\]
Note that the gap provided by $\delta$ is crucial for this to hold: the same assumptions but with $\delta=0$ are \emph{not} sufficient.

We will call the corresponding parameters $p$ and $\sigma$ respectively the position and scale parameters of $Z$, and write them $p_Z$ and $\sigma_Z$ (for this, we do not need to assume that $Z$ is Cauchy distributed, but only that it lies in the strong domain of attraction of a Cauchy distribution).

%
%

\section{Former results}
We will use the following results from \cite{LAWA} and \cite{BWE}.

\begin{lemma}[{\cite[Lemma 5.2]{LAWA}}]
  \label{le:position}
  Assume $Z$ belongs to the strong attraction domain of a Cauchy distribution. Then, its position parameter $p_Z$ is equal to
  \[\lim_{N\to \infty} \mathbb{E}[ \cut{N}{Z}  ].\]
\end{lemma}

When $Y$ and $Z$ lie in the strong attraction domain of Cauchy distributions, or even when they are Cauchy random variables, but they are not independent, $Y+Z$ does not necessarily belong to the strong attraction domain of a Cauchy distribution.
What might be even more surprising is that, even if $Y$, $Z$, and $Y+Z$ are Cauchy random variables, $p_{Y+Z}$ may differ from  $p_Y+p_Z$ (see e.g. \cite{chenShepp} for a constructive counter-example). Yet, the following lemma shows that, to ensure additivity, it suffices that the tail behavior of $Y$ and $Z$ are not too strongly correlated.
\begin{lemma}[{\cite[Lemma 5.3]{LAWA}}]
  \label{le:preline}
  Let $n\geq 1$ and $Z_1,\dots, Z_n$ be random variables which each lie in the strong attraction domain of a Cauchy distribution.
  Assume that there exists $\delta>0$ such that, for all $i,j\in \{1,\dots, n\}$, $i\neq j$,
  \begin{equation}
  \label{eq:cond:decor}
  \mathbb{P}( |Z_i|\geq r \mbox{ and } |Z_j|\geq r)\underset{r\to +\infty}=o(r^{-(1+\delta)}).
  \end{equation}
  Then, $Z\coloneqq \sum_{i=1}^n Z_i$ also lies in the strong attraction domain of a Cauchy distribution, and $p_Z=\sum_{i=1}^n p_{Z_i}$.
\end{lemma}
The reason why we consider these strong  attraction domain is the following. Consider a convex set $K$, and $X$ a planar Brownian motion. Let $P$ be distributed uniformly at random inside $K$, independently from $X$. Then, $X$-almost surely in the event that $\Range(\bar{X})\subset K$, the random variable $f(P)\n_X(P)$ (recall the realisation of $X$ is fixed but $P$ is random) lies in the strong attraction domain of a Cauchy distribution, which is essentially equivalent to the following lemma.
\begin{lemma}[{\cite[Lemma 5]{BWE}}]
  \label{le:bound}
  Let $X:[0,1]\to \mathbb{R}^2$ be a planar Brownian motion.
  For all $\beta'<\frac{1}{2}$, there exists $\delta>0$ such that almost surely, there exists $C$ such that for all bounded and uniformly continuous function $f\in \mathcal{C}_b(\R^2)$, for all $n\geq 1$,
  \begin{equation}
  \label{eq:fbound1}
  \Big|2\pi n f(\mathcal{D}_n)-\int_0^1 f(X_u)\d u \Big| \leq C (\omega_f (2|X|_{\mathcal{C}^{\beta'}} n^{-\delta} )  +\|f\|_\infty n^{-\delta} ),
  \end{equation}
  where $\omega_f$ is the continuity modulus of $f$, i.e. $\omega_f(\epsilon)\coloneqq \sup_{x,y: |x-y|\leq \epsilon} |f(x)-f(y)|$.
\end{lemma}
Of course, from symmetry in distribution of the Brownian motion, Lemma~\ref{le:bound} also holds when $\mathcal{D}_n$ is replaced with $\mathcal{D}_{-n}$.

We will also need some control in $L^p(\Omega)$.
\begin{lemma}[{\cite[Theorem 6.2]{LAWA}}]
  \label{le:lpLAWA}
  For all $\delta<\frac{1}{2}$ and $p\geq 2$, there exists a constant $C$ such that for all $N\geq 1$,
  \[
  \mathbb{E}\big[ \big|D_N-\tfrac{1}{2\pi N}\big|^p\big]^{\frac{1}{p}}\leq CN^{-1-\delta}.\]
\end{lemma}

Finally, the following lemma will be used to check the asymptotic decorrelation condition~\eqref{eq:cond:decor}.
\begin{lemma}[{\cite[Lemma 2.4 with $N=M$ and $T=2$]{LAWA}}]
\label{le:joint}
  Let $X,X':[0,1]\to \mathbb{R}^2$ be two independent Brownian motions in the plane.
   Then, for all $p\in[1,\infty)$, there exists $C_p<\infty$, uniform on the starting points of $X,X'$, such that for all $n>1$,  \[\mathbb{E}[  |\mathcal{D}^X_n\cap \mathcal{D}^{X'}_{n} |^p]\leq C_p\log(1+n)^{3p+1}  n^{-2p}  .\]
\end{lemma}
\begin{corollary}
\label{coro:joint}
  Let $X,X':[0,1]\to \mathbb{R}^2$ be two independent Brownian motions in the plane, with possibly random starting points.
  Then, for all $\epsilon>0$, almost surely, $n^{2-\epsilon}  |\mathcal{D}^X_n\cap \mathcal{D}^{X'}_{n} |\underset{n\to \infty}\longrightarrow 0$.
\end{corollary}
\begin{proof}
  Let $p>\max(1,\epsilon^{-1})$. Let $\delta>0$.
  Then,
  \begin{align*}
  \mathbb{P}(\exists n\geq n_0 :   n^{2-\epsilon} |\mathcal{D}^X_n\cap \mathcal{D}^{X'}_{n}| \geq \delta )
  &\leq
  \sum_{n\geq n_0}\delta^{-p} \mathbb{E}\big[   (n^{2-\epsilon} |\mathcal{D}^X_n\cap \mathcal{D}^{X'}_{n}|)^p\big]\\
  &\leq
  \delta^{-p} \sum_{n\geq n_0} n^{2p-\epsilon p} C_p\log(1+n)^{3p+1}  n^{-2p}\\
  &=C_p\delta^{-p} \sum_{n\geq n_0} \log(1+n)^{3p+1} n^{-\epsilon p}\underset{n_0\to \infty}\longrightarrow 0,
  \end{align*}
  which concludes the proof.
\end{proof}

\section{Stokes formula}
In this section, $X:[0,1]\to \mathbb{R}^2$ is a standard Brownian motion under $\mathbb{P}$.
\subsection{Existence of a limit}
We shall first prove the first part of Theorem~\ref{th:green}:
\begin{lemma}
\label{le:limitExists}
  Let $\beta>0$. $\mathbb{P}$-almost surely, for all $f\in \mathcal{C}_2^\beta ( \mathbb{R}^2)$, the limits
  \[ \dashint \n_X f \d \lambda \coloneqq\lim_{N\to \infty} \int_{\mathbb{R}^2} \cut{N}{\n_X(z)}f(z)\d z \qquad \mbox{and}\qquad   \lim_{N\to \infty} \int_{\mathbb{R}^2} \n_{X}(z) \mathbbm{1}_{|\n_{X}(z)|\leq N}\  f(z)\d z\]
  exist and are equal. Furthermore, almost surely, the application $f\mapsto \dashint \n_Xf \d \lambda $, from $\mathcal{C}^\beta_2(\mathbb{R}^2)$ endowed with the norm $\|\cdot \|_{ \mathcal{C}^\beta_\infty }$ to $\mathbb{R}$, is continuous.
\end{lemma}
\begin{proof}
  Let $\beta'\in\big(0,\tfrac{1}{2}\big)$. Let
  $\delta>0$ and $C(X)$ the random constant from Lemma~\ref{le:bound}. Taking $C\coloneqq \max(C(X), C(\hat{X}))$ where $\hat{X}$ is a reflection of $X$, say along the $x$-axis, we see (by taking also the reflection of $f$) that~\eqref{eq:fbound1} holds also when we replace $\mathcal{D}_n$ with $\mathcal{D}_{-n}$.
  Thus, for all $f\in \mathcal{C}^\beta_2 ( \mathbb{R}^2)$,
  \begin{align}
  \big| f(\mathcal{D}_n)-f(\mathcal{D}_{-n}) \big|
  &\leq \frac{1}{2\pi n} \Big(   \Big| 2\pi n f(\mathcal{D}_n) -  \int_0^1 f(X_u)\d u|+ \Big|   \int_0^1 f(X_u)\d u-2\pi n f(\mathcal{D}_{-n}) \Big|\Big)\nonumber\\
  &\leq \frac{C}{ \pi n}\big(\omega_f \big(2|X|_{\mathcal{C}^{\beta'}} n^{-\delta} \big)  +\|f\|_\infty n^{-\delta} \big)\nonumber \\
  &\leq \frac{C   }{ \pi n} \big( |f|_{\mathcal{C}^{\beta} },  |X|_{\mathcal{C}^{\beta'}}^\beta n^{-\delta\beta}  +\|f\|_\infty n^{-\delta} \big)=O(n^{-1-\delta\beta}). \label{eq:bound}
  \end{align}

  Thus, on the almost sure event $\mathcal{E}\coloneqq \{|X|_{\mathcal{C}^{\beta'}}<\infty$, $C<\infty\}$, which does not depend on $f$, the sum
  \[
  \sum_{n\geq 1}   (f(\mathcal{D}_n)-f(\mathcal{D}_{-n}) )
  \]
  is absolutely convergent.
  By summation by part,
  \[\sum_{n=1}^N   (f(\mathcal{D}_n)-f(\mathcal{D}_{-n}) )
  =\sum_{k=1}^\infty \cut{N}{k} (f(\mathcal{A}_k)-f(\mathcal{A}_{-k}))
  = \int_{\mathbb{R}^2} \cut{N}{\n_X(z)}f(z)\d z,  \]
  so that the right-hand side is convergent on the event $\mathcal{E}$.

\smallskip
  Consider now the other cut-off. Remark
  \[ \Big| \int_{\R^2} \cut{N}{\n_X(z)} f(z)\d z  -  \int_{\R^2}  \n_X(z) \mathbbm{1}_{|\n_X(z)|\leq N} f(z)\d z \Big| = N|f(\mathcal{D}_{N+1})-f(\mathcal{D}_{-N-1}) |, \]
  which, on the almost sure event $\mathcal{E}$, converges toward $0$ as $N$ goes to infinity (by~\eqref{eq:bound}), which proves indeed that the second limit is also well-defined and equal to the first one.

  It remains to prove the almost sure continuity of the application $f\mapsto  \dashint \n_X f \d \lambda$. Since it is clearly a linear application, it suffices to show that it is almost surely a bounded one (in the sense of linear operators). By~\eqref{eq:bound},
  \[
  \sum_{n= 1}^N  | f(\mathcal{D}_n)-f(\mathcal{D}_{-n})|\leq \frac{C (1+ |X|_{\mathcal{C}^{\beta'}}^\beta  ) }{\pi}\Big(\sum_{n=1}^\infty n^{-1-\delta\beta}\Big)
  \max(   |f|_{\mathcal{C}^{\beta} }, \|f\|_\infty  )
  \leq  C'\|f\|_{\mathcal{C}^\beta_\infty} ,
  \]
  for a random constant $\mathcal{C}'$ which depends on $\beta,\ \beta'$ and $\delta$, but not on $f$ nor $N$. Letting $N\to \infty$, we deduce \[| \dashint \n_Xf \d \lambda|\leq C' \|f\|_{\mathcal{C}^\beta_\infty},\] which concludes the proof.
\end{proof}

\subsection{Strategy for the Green formula}
We now wish to identify $ \dashint \n_Xf\d \lambda$
with the Stratonovich integral $ \int_X \eta+ \int_{[X_1,X_0]} \eta$,  when $f=\partial_1 \eta_2-\partial_2\eta_1$.
To this end, we decompose the trajectory $X$ as follows.
First, we denote by $X^{(n)}$ the dyadic piecewise-linear approximation of $X$ with $2^n$ steps: for $\lambda \in [0,1],\ i\in \{0,\dots, 2^{n}-1\}$, and $t= (i+\lambda)2^{-n}$,
\[ X^{(n)}_t\coloneqq X_{i2^{-n}}+\lambda ( X_{(i+1)2^{-n}} -X_{i2^{-n}} ).\]
For $i\in \{0,\dots, 2^{n}-1\}$, we also set $X^i$,  the restriction of $X$ to the interval $[i2^{-n},(i+1)2^{-n}]$.
Since $X^i$ is also a Brownian motion, the almost sure limit
\[ \dashint \n_{X^i}f\d \lambda \coloneqq \lim_{N\to \infty}\int_{\R^2}  \cut{N}{\n_{X^i}(z)}   f(z)\d z\]
is well-defined by Lemma~\ref{le:limitExists}, for all $f\in \mathcal{C}^\beta_2(\mathbb{R}^2)$. The fact that $X^i$ starts from a random point is not an issue: it suffices to apply Lemma~\ref{le:limitExists} to the random function $z\mapsto f(z+X_{i2^{-n}})$ and rely on translation invariance. Crucially,
Lemma~\ref{le:limitExists} identifies an almost sure event such that $\dashint \n_{X}f\d \lambda$ is well-defined \emph{for all $f$ at once}, so that we can indeed apply it to a random function.

On $\mathbb{R}^2\setminus (\Range(X)\cup \Range(X^{(n)}))$, the additivity of the winding index, with respect to concatenation of loops, gives the identity
\[
\n_X=\n_{X^{(n)}}+\sum_{i=0}^{2^n-1} \n_{X_i}.
\]
We shall prove that under mild conditions, $\dashint$ behaves additively : almost surely, for all $f\in \mathcal{C}^\beta_2(\mathbb{R}^2)$ and $n\geq 1$,
\[
\dashint \n_X f\d \lambda=\sum_{i=0}^{2^n-1} \dashint \n_{X_i} f\d  \lambda+ \int_{\mathbb{R}^2} \n_{X^{(n)}} f\d  \lambda=\sum_{i=0}^{2^n-1} \dashint \n_{X_i}(z) f(z)\d z+ \int_{X^{(n)}}  \eta+\int_{[X_1,X_0]} \eta.
\]
The second equality is by the standard Stokes' formula, for piecewise-linear loops.
As $n$ goes to infinity, we will see that the sum over $i$ in these expressions asymptotically vanishes, and that the integral along  $X^{(n)}$ converges toward the Stratonovich integral $\int_{X}  \eta$, which gives the expected formula.

Thus, in order to prove~\eqref{eq:thgreensto} and thus conclude the proof of Theorem~\ref{th:green}, it suffices to prove the three following lemma. Let $f\in \mathcal{C}^\beta_2(\mathbb{R}^2)$, and $\eta\in \mathcal{C}^{1+\beta}_2(\mathbb{R}^2,\mathbb{R}^2)$ such that $f=\partial_1\eta_2-\partial_2\eta_1$.
\begin{lemma}
  \label{le:line}
  For all $n$, almost surely,
  \begin{equation}
  \dashint \n_X f\d \lambda=\sum_{i=0}^{2^n-1} \dashint \n_{X_i} f\d \lambda+ \int_{\mathbb{R}^2} \n_{X^{(n)}} f\d \lambda.
  \end{equation}
\end{lemma}
\begin{lemma}
  \label{le:residu}
  As $n$ goes to infinity,
  \[\sum_{i=0}^{2^n-1}\dashint \n_{X_i}f\d \lambda\]
  converges almost surely toward zero.
\end{lemma}
\begin{lemma}
  \label{le:strato}
  As $n$ goes to infinity, the integral $ \int_{X^{(n)}} \eta$ converges almost surely toward the Stratonovich integral $\int_X \eta $.
\end{lemma}

\subsection{Proof of Lemma~\ref{le:line}}
Intuitively, the equality in Lemma~\ref{le:line} follows from integration of the almost-everywhere equality
\[ \n_X(z) =\sum_{i=0}^{2^n-1} \n_{X^i}(z)+ \n_{X^{(n)}}(z),\]
applied together with the Stokes formula for $X^{(n)}$.
However, neither $\n_X$ nor $\n_{X^i}$ are integrable, and we have to deal with the cut-offs which allowed to define $\dashint \n_Xf\d \lambda$ and the $\dashint \n_{X^i}f\d \lambda$~: for any finite $N$, the terms
\[
\cut{N}{\n_X(z)} \qquad \text{and}\qquad   \sum_{i=0}^{2^n-1} \cut{N}{\n_{X^i}(z) }+ \cut{N}{\n_{X^{(n)}}(z)}
\]
are not necessarily equal.

\begin{proof}[Proof of Lemma~\ref{le:line}]
  From linearity with respect to $f$, we can and we do assume $f\geq 0$. In the event that that the restriction of $f$ to the ball $B(0,\|X\|_{\infty})$ is identically vanishing, the result is trivial, and we thus assume that \[Z\coloneqq \int_{ B(0,\|X\|_{\infty})} f(z)\d z\]
  is strictly positive.
  Let $P$ be a random point in $\mathbb{R}^2$, those distribution conditional on $X$ admits a density with respect to the Lebesgue measure, given at $z\in \mathbb{R}^2$ by
  \[
  Z^ {-1}f(z) \mathbbm{1}_{ B(0,\|X\|_{\infty})   }(z),
  \]
  Let $\mathbb{Q}\coloneqq\mathbb{P}(\ \cdot \ |X)$.
  Then, $X$-almost surely, $\mathbb{Q}$-almost surely,
  \[\n_{X}(P) = \sum_{i=0}^{2^n-1} \n_{X_i}(P) + \n_{X^{(n)}}(P)  .\]
  For $N\geq 0$, for $\tilde{X}\in\{ X, X^i : i\in \{0,\dots 2^n-1\}\}$, it holds
  \[
  \mathbb{Q}(\n_{\tilde{X}}(P)\geq N)= Z^ {-1}f(\mathcal{D}^{\tilde{X}}_{N}) , \qquad \mathbb{Q}(\n_{\tilde{X}}(P)\leq - N)= Z^ {-1} f(\mathcal{D}^{\tilde{X}}_{-N}).
  \]
  Thus, Lemma~\ref{le:bound} (with the appropriate scaling in the case $\tilde{X}\in \{ X^i : i\in \{0,\dots 2^n-1\}\}$) ensures that for some $\delta>0$, $X$-almost surely,
  \[
  \mathbb{Q}(\n_{\tilde{X}}(P)\geq N)= \frac{Z^ {-1}}{2\pi N}\int f(\tilde{X}_u)\d u +O(N^{-1-\beta\delta}),
  \]
  and $  \mathbb{Q}(\n_{\tilde{X}}(P)\leq -N)= \frac{Z^ {-1}}{2\pi N}\int f(\tilde{X}_u)\d u +O(N^{-1-\beta\delta})$.
  Thus, $X$-almost surely, the random variable $\n_{\tilde{X}}(P)$ belong to the strong attraction domain of a Cauchy distribution.
  Furthermore, $|\n_{X^{(n)}}|$ is bounded by $2^n$ and therefore $\n_{X^{(n)}}(P)$ also belong to the strong attraction domain of a (degenerate, i.e. with $\sigma=0$) Cauchy distribution.

  In order to  apply Lemma~\ref{le:preline} to the set of variables \[(Z_0,\dots,Z_{2^n-1},Z_{2^n})\coloneqq(\n
  _{X^0}(P),\dots, \n_{X^{2^n-1}}(P), \n_{X^{(n)}}(P)),\]
  we need to check that the asymptotic decorrelation condition~\eqref{eq:cond:decor} holds under $\mathbb{Q}$.
  First, for $i\in \{0,\dots, 2^n-1\}$, for $x\geq 2^n$, since $|\n_{X^{(n)}}(P )|\leq 2^n$,
  \[ \mathbb{Q}( |\n_{X^i}(P )|\geq x \mbox{ and }   |\n_{X^{(n)}}(P )|\geq x  )=0=o(x^{-(1+\delta)}).\]
  Besides, by Corollary~\ref{coro:joint},
  for $i,j\in \{0,\dots, 2^n-1\}$, $i\neq j$,
  \begin{align*}
  \mathbb{Q}( |\n_{X^i}(P )|\geq N \mbox{ and }   |\n_{X^j}(P )|\geq N  )
  &=
  \frac{1}{Z}  f\big( \big( \mathcal{D}_N^{X^i}\cup \mathcal{D}_{-N}^{X^i}\big)\cap  \big( \mathcal{D}_N^{X^j}\cup \mathcal{D}_{-N}^{X^j}\big)\big)\\
  &\leq
  \frac{\|f\|_\infty}{Z}  \big|\big(\mathcal{D}_N^{X^i}\cup \mathcal{D}_{-N}^{X^i}\big)\cap  \big( \mathcal{D}_N^{X^j}\cup \mathcal{D}_{-N}^{X^j}\big) \big|\\
  &\leq \frac{4 C \|f\|_\infty}{Z} |N|^{-2+\epsilon},
  \end{align*}
  for an arbitrary $\epsilon>0$ (say $\epsilon=1/2$, so that $-2+\epsilon<-1$), and
  for $C$ the corresponding constant from Corollary~\ref{coro:joint}, applied to the independent Brownian motions
  \[ \hat{X}_i: t\mapsto X_{(i+1-t)2^{-n}}- X_{(i+1)2^{-n}},\qquad  \hat{X}_j:t \mapsto
  X_{(j+t)2^{-n}}-X_{(i+1)2^{-n}}\]
(we also gain a factor $2^{-n}$ from scaling, but this does not matter here as $n$ is fixed).

  Thus, the condition~\eqref{eq:cond:decor} is satisfied and we can apply Lemma~\ref{le:preline} to deduce that, $X$-almost surely, the position parameters of the variables $\{Z_i: i\in \{0,\dots, 2^n\} \}$ add up:
  \begin{equation}
  \label{eq:pAdd}
  p_{\n_X(P)}
  =p_{\sum_{i=0}^{2^n} Z_i}
  =\sum_{i=0}^{2^n} p_{Z_i}
  =\sum_{i=0}^{2^n-1} p_{\n_{X^i}(P)}+p_{\n_{X^{(n)} }(P) }.\end{equation}
  Furthermore, since $|\n_{X^{(n)} }(P)|$ is bounded, $X$-almost surely,
  \[ p_{\n_{X^{(n)} }(P) }=\mathbb{E}^\mathbb{Q}[ \n_{X^{(n)} }(P)  ]=\frac{1}{Z} \int_{\mathbb{R}^2}  \n_{X^{(n)} }(z) f(z)\d z.
  \]

  Finally, from Lemma~\ref{le:position}, we deduce that $X$-almost surely,
  \begin{align*}
  p_{n_X(P)}=  \lim_{N\to \infty} \mathbb{E}^\mathbb{Q}\big[ \cut{N}{\n_X(P)}    \big]=  \lim_{N\to \infty}  \frac{1}{Z}\int_{\mathbb{R}^2}\cut{N}{\n_X(P)}   f(z)\d z=\frac{1}{Z}\dashint \n_Xf\d \lambda,
  \end{align*}
  and similarly
  \[ p_{n_{X^i}(P)}= \frac{1}{Z}\dashint \n_{X^i}f\d \lambda.\]
  Inserting the three last equations inside~\eqref{eq:pAdd} gives
  \[
  \dashint \n_Xf\d \lambda=\sum_{i=0}^{2^n-1}\dashint \n_{X^i}f\d \lambda+ \int_{\mathbb{R}^2}  \n_{X^{(n)} }(z) f(z)\d z,
  \]
  which concludes the proof.
\end{proof}

\subsection{Proof of Lemma~\ref{le:residu}}
In order to prove  Lemma~\ref{le:residu}, we first need the following result, which roughly states that when restricted to functions $f\in \mathcal{C}^\beta_2(\mathbb{R}^2)$, the constant $C$  in Lemma~\ref{le:bound} can be chosen to lie in $L^p(\Omega)$.
\begin{lemma}
  \label{le:boundLp}
  Let $\beta>0$ and $p\geq 1$. There exist a constant $C$ and $\delta>0$ such that for all $f\in \mathcal{C}^\beta_2(\mathbb{R}^2)$ and all $N\geq 1$,
  \[
  \mathbb{E}\big[\big|f(\mathcal{D}^X_N) -\frac{1}{2\pi N} \int_0^1 f(X_t)\d t \big|^p\big]^{\frac{1}{p}}\leq C N^{-1-\delta}\|f\|_{\mathcal{C}^\beta_\infty}.
  \]
\end{lemma}
\begin{proof}
  The proof is largely inspired from \cite{BWE}. During this proof, we consider a decomposition of $X$ similar to the one in the rest of the paper, but into $T$ pieces, for $T$ a positive integer which is a function $N$. Thus, for $i\in \{0,\dots,T-1\}$, we let $Y^i$ be the restriction of $X$ to the interval $[iT^{-1},(i+1)T^{-1}]$.

  Decomposing $f=f_+-f_-$ and eventually replacing $C$ with $2C$, we can again assume that $f\geq 0$. Let $X^{pl}$ be the piecewise-linear approximation of $X$ with $T$ pieces,
  \[ X^{pl}_{(i+\lambda)T^{-1}}\coloneqq X_{iT^{-1}}+\lambda (X_{(i+1)T^{-1}}- X_{iT^{-1}}) \quad \text{for} \quad i\in \{0,\dots, T-1\}\text{ and } \lambda\in [0,1].\]
  For $i,j\in \{0,\dots, T-1\}$, let
  \[ \mathcal{D}^i_N\coloneqq\mathcal{D}^{Y^i}_N,  \qquad \mathcal{D}^{i,j}_N\coloneqq\big(\mathcal{D}^{Y^i}_N \cup \mathcal{D}^{Y^i}_{-N} \big)\cap \big(\mathcal{D}^{Y^j}_N \cup \mathcal{D}^{Y^j}_{-N} \big). \]
  For $z$ outside $\Range(X)\cup \Range(X^{pl})$, we have
  \[ \n_X(z)=\sum_{i=0}^{T-1} n_{Y^i}(z)+ \n_{X^{pl}}(z) \quad\text{and}\quad |\n_{X^{pl}}(z)|\leq T. \]
Thus, for $\n_X(z)$ to be larger than $N$, it must hold that $\sum_{i=0}^{T-1} \n_{Y^i}(z)\geq N-T$. For an arbitrary $M\geq 0$, this can happen in two ways: either  there exists $i$ such that $n_{Y^i}(z)\geq N-T-M(T-1)$, or there exists $i,j$ with $i\neq j$ such that both $n_{Y^i}(z)\geq M$ and $n_{Y^j}(z)\geq M$. Thus, for all $T,M,N\geq 1$ such that $T(M+1)<N$, we have the deterministic inclusion
  \[
  \mathcal{D}^X_N\subseteq \bigcup_{i=0}^{T-1} \mathcal{D}^i_{N-T-M(T-1)} \cup \bigcup_{\substack{i,j=0\\i\neq j}}^{T-1} \mathcal{D}^{i,j}_M \cup \Range(X)\cup \Range(X^{pl}),
  \]
  and therefore, since $f\geq 0$,
  \[
  f(\mathcal{D}^X_N)\leq  \sum_{i=0}^{T-1} f(\mathcal{D}^i_{N-T-M(T-1)})+  \sum_{\substack{i,j=0\\i\neq j}}^{T-1} f(\mathcal{D}^{i,j}_M).
  \]
  We set $t\in(0,\tfrac{1}{3})$, $m\in (\tfrac{1+t}{2}, 1-t)$, $\beta'<\frac{1}{2}$,  $T\coloneqq \lfloor N^t\rfloor$, $M\coloneqq \lfloor N^m\rfloor$, and we assume that $N$ is large enough for the inequality  $T(M+1)<N$ to hold. We also set $N'\coloneqq  N-T-M(T-1)$ to ease notations.

  Remark $\mathcal{D}^i_{N'}$ is contained in the ball centered at $X_{\frac{i}{T}}$ with radius $|X|_{\mathcal{C}^{\beta'}} T^{-\beta'}$. Thus, for all $z\in \mathcal{D}^i_{N'}$,
  $|f(z)-f(X_{iT^{-1}})|\leq |f|_{\mathcal{C}^\beta} |X|^{\beta}_{\mathcal{C}^{\beta'}} T^{-\beta\beta' }  $.
  Thus,
  \begin{align*}
  f(\mathcal{D}^X_N)&\leq
   \sum_{i=0}^{T-1} f(X_{iT^{-1}} )  |\mathcal{D}^i_{N'}| + |f|_{\mathcal{C}^\beta} |X|^{\beta}_{\mathcal{C}^{\beta'}} T^{-\beta \beta'}  \sum_{i=0}^{T-1}  D^i_{N'}+ \|f\|_\infty \sum_{i\neq j} D^{i,j}_M.\\
   &\leq\Big(
   \frac{1}{2\pi N T}\sum_{i=0}^{T-1} f(X_{iT^{-1}} )+  \|f\|_\infty\sum_{i=0}^{T-1}  \big|\tfrac{1}{2\pi N T}- D^i_{N'}\big| \Big)+ |f|_{\mathcal{C}^\beta} |X|^{\beta}_{\mathcal{C}^{\beta'}} T^{-\beta \beta'}  \sum_{i=0}^{T-1}  D^i_{N'}\\
   & \hspace{10.9cm}+ \|f\|_\infty \sum_{i\neq j} D^{i,j}_M\\
   &\leq \frac{1}{2\pi N}\int_0^1 f(X_t)\d t + \frac{|f|_{\mathcal{C}^\beta} |X|_{\mathcal{C}^{\beta'}}^\beta T^{-\beta \beta'} }{2\pi N} + \|f\|_\infty\sum_{i=0}^{T-1}  \big|\tfrac{1}{2\pi N T}- D^i_{N'}\big| \\
   &\hspace{6cm}+|f|_{\mathcal{C}^\beta} |X|^{\beta}_{\mathcal{C}^{\beta'}} T^{-\beta \beta'}  \sum_{i=0}^{T-1}  D^i_{N'}+ \|f\|_\infty \sum_{i\neq j} D^{i,j}_M.
  \end{align*}
  Writing $g^p_+$ for the $p$-th power of the positive part of a function $g$, and using the triangle inequality in $L^p(\mathbb{P})$, we obtain
  \begin{align*}
  \mathbb{E}\Big[ \Big( f(\mathcal{D}_N^X) -& \frac{1}{2\pi N}  \int_0^1 f(X_t)\d t \Big)_+^p \Big]^{\frac{1}{p}}
  \leq  \frac{|f|_{\mathcal{C}^\beta} T^{-\beta \beta'}  }{2\pi N} \mathbb{E}[|X|_{\mathcal{C}^{\beta'}}^{\beta p} ]^{\frac{1}{p}}
  +\|f\|_\infty \mathbb{E}\Big[ \Big|\frac{1}{2\pi N} - D_{N'}^X \Big|^p \Big]^{\frac{1}{p}}\\
  &+|f|_{\mathcal{C}^\beta} T^{-\beta \beta'}\mathbb{E}[  (D_{N'}^X)^{2p} ]^{\frac{1}{2p}}\mathbb{E}[ |X|_{\mathcal{C}^{\beta'}}^{2p\beta}]^{\frac{1}{2p}}
  + \|f\|_\infty \mathbb{E}\Big[\Big(\sum_{i\neq j} D^{i,j}_M \Big)^p\Big]^{\frac{1}{p}}.
  \end{align*}
  We now use the asymptotic equivalences $N'\sim N$ and $\frac{1}{N}-\frac{1}{N'}\sim  N^{t+m-2}$, as $N\to\infty$, as well as Lemma~\ref{le:lpLAWA}, and the following estimations (\cite[Lemma 2.4]{LAWA}): for all $p\geq 1$, there exists a constant $C$ such that for all $N\geq 1$,
  \[
  \mathbb{E}\Big[\Big(\sum_{i\neq j} D^{i,j}_M \Big)^p\Big]^{\frac{1}{p}}\leq C \log(N+1)^{3+\frac{2}{p}} M^{-2} T^{1-\frac{1}{p}}.
  \]
  We end up with
  \begin{align*}
  \mathbb{E}\Big[ \Big( f(\mathcal{D}_N^X) - \frac{1}{2\pi N}  \int_0^1 f(X_t)\d t \Big)_+^p \Big]^{\frac{1}{p}}
  &\leq C \big( |f|_{\mathcal{C}^\beta} N^{-1-t\beta \beta'}+\|f\|_\infty N^{m+t-2}
  +\|f\|_\infty N^{-1-\delta }\\
  &+|f|_{\mathcal{C}^\beta} N^{-1-t \beta \beta'}
  + \|f\|_\infty \log(N+1)^{3+\frac{2}{p}} N^{-2 m+t-\frac{t}{p}}\big),
  \end{align*}
  for $\delta\in(0,\frac{1}{2})$ arbitrary.
  The conditions on $t$ and $m$ ensure that the powers of $N$ in the right-hand side are all less than $-1$, so that there exists $\delta'$ and $C$ such that
  \[  \mathbb{E}\Big[ \Big( f(\mathcal{D}_N^X) - \frac{1}{2\pi N}  \int_0^1 f(X_t)\d t \Big)_+^p \Big]^{\frac{1}{p}}\leq C\|f\|_{\mathcal{C}_\infty^{\beta'}} N^{-1-\delta'}. \]
  The negative part is treated in a similar way, using instead the inclusion
  \[
  \mathcal{D}^X_N\supseteq \bigcup_{i=0}^{T-1} \mathcal{D}^i_{N+T+M(T-1)} \setminus \Big( \bigcup_{\substack{i,j=0\\i\neq j}}^{T-1} \mathcal{D}^{i,j}_M \cup \Range(X)\cup \Range(X^{pl})\Big),
  \]
  and the lemma follows.
\end{proof}

\begin{corollary}
\label{coro:lpbound}
  Let $\beta>0$ and $p\geq 1$. There exists a constant $C$ such that for all $f\in \mathcal{C}_2^\beta(\mathbb{R}^2)$,
  $\mathbb{E}[|\dashint \n_Xf\d \lambda |^p]^{\frac{1}{p}}\leq C\|f\|_{\mathcal{C}_\infty^\beta}$.
\end{corollary}
\begin{proof}
  Let $C$ and $\delta $ from Lemma~\ref{le:boundLp}. Then, for all $f\in \mathcal{C}_2^\beta$ and $n\geq 1$,
  \[ \mathbb{E}\big[|f(\mathcal{D}^X_n)-f(\mathcal{D}^X_{-n})|^p\big]^{\frac{1}{p}}\leq 2C n^{-1-\delta} \|f\|_{\mathcal{C}_2^\beta}.\]
  As previously remarked, $\dashint \n_Xf\d \lambda=\sum_{n=1}^\infty (f(\mathcal{D}^X_n)-f(\mathcal{D}^X_{-n}))$ by summation by part.
  By triangle inequality in $L^p(\mathbb{P})$,
  \[ \mathbb{E}[|\dashint \n_Xf\d \lambda |^p]^{\frac{1}{p}}=\mathbb{E}\Big[\Big|\sum_{n=1}^\infty (f(\mathcal{D}^X_n)-f(\mathcal{D}^X_{-n})) \Big|^p\Big]^{\frac{1}{p}}
  \leq 2C \|f\|_{\mathcal{C}_\infty^\beta} \sum_{n=1}^\infty N^{-1-\delta}
  \leq C'   \|f\|_{\mathcal{C}_\infty^\beta}.\qedhere\]
\end{proof}
We now prove that almost surely,
\[\sum_{i=0}^{2^n-1}\dashint \n_{X^i}f\d \lambda \underset{n \to\infty}\longrightarrow 0.\]
\begin{proof}[Proof of Lemma~\ref{le:residu}]
  For $i\in \{0, \dots, 2^n-1\}$, we define $\bar{f}^i:\R^2 \to \R$ the constant function equal to $f(X_{i2^{-n}} )$, and $\tilde{f}^i\coloneqq f-\bar{f}^i$. Since for all $i$, $f\mapsto \dashint_{X_i} \n_X f\d \lambda$ is a linear map, it suffices to show that both
  \[\sum_{i=1}^{2^n}  \dashint \n_{X_i} \bar{f}^i\d \lambda=\sum_{i=1}^{2^n} f(X_{i2^{-n}} )  \dashint \n_{X_i}\d \lambda   \qquad\mbox{ and }\qquad\sum_{i=1}^{2^n}  \dashint \n_{X_i} \tilde{f}^i\d \lambda\]
  almost surely converge toward $0$ as $n\to \infty$. This allows to use different types of arguments: for the first term we can rely on symmetry and use the fact that $\dashint \n_{X_i}\d \lambda$ vanishes on average, while for the second term we can use the fact that $\tilde{f}^i$ vanishes at $X_{i2^{-n}} $ and remains small in its vicinity.

  From symmetry, for all $i$, almost surely,
  $\mathbb{E}\Big[\dashint \n_{X_i} \d \lambda\Big|(X_s)_{s\leq i2^{-n} }\Big]=0$. It follows that, for $i<j$,
  \[\mathbb{E}\Big[f(X_{i2^{-n}}) f(X_{j2^{-n}}) \dashint_{X_i} \n_{X_i}\d \lambda \dashint  \n_{X_j} \d \lambda \Big]=
  0.\]
  Besides, from a simple scaling argument,
  \[\mathbb{E}\Big[\Big( \dashint  \n_{X_i} \d \lambda\Big)^2\Big]= 2^{-2n} \mathbb{E}\Big[\Big(\dashint \n_X\d \lambda \Big)^2\Big],\]
where $\mathbb{E}[(\dashint \n_X \d \lambda)^2]<\infty$, which follows for example from Corollary~\ref{coro:lpbound}.

  We deduce that
  \begin{align*}
  \mathbb{E}\Big[ \Big( \sum_{i=1}^{2^n} \dashint \n_{X_i} \bar{f}^i  \d \lambda \Big)^2 \Big]
  & =
\sum_{i=1}^{2^n} \mathbb{E}\Big[f(X_{i2^{-n}} )^2 \Big(\dashint \n_{X_i}   \d \lambda \Big)^2 \Big]
  \\
  &\leq 2^{-n} \|f\|^2_\infty \mathbb{E}\Big[\Big(\dashint \n_{X}\d \lambda\Big)^2\Big].
  \end{align*}
  Since this is summable over $n$, the almost sure convergence follows: for all $\epsilon>0$,
  \begin{align*}
  \mathbb{P}\Big( \exists n\geq n_0:  \Big| \sum_{i=1}^{2^n} \dashint \n_{X_i} \bar{f}^i  \d \lambda \Big|\geq \epsilon\Big)
  &\leq \frac{1}{\epsilon^2}  \mathbb{E}\Big[ \sup_{n\geq n_0}      \Big( \sum_{i=1}^{2^n} \dashint \n_{X_i} \bar{f}^i \d \lambda \Big)^2 \Big]\\
  &\leq \frac{ \sum_{n=n_0}^\infty 2^{-n}}{\epsilon^2} \|f\|^2_\infty \mathbb{E}\Big[\Big(\dashint \n_X\d \lambda\Big)^2\Big]\\
  &\underset{n_0\to\infty}\longrightarrow 0.
  \end{align*}

  In order to deal with the sum involving $\tilde{f}^i$, one must be careful in the way we use the translation invariance and scale invariance of the Brownian motion, the reason being that the norm $\| \cdot \|_{\mathcal{C}^\beta_\infty}$ is inhomogeneous under rescaling (i.e. under precomposition by multiplication with a scalar). Let $\beta'\in(0,\frac{1}{2})$, and for a given $R\geq 1$, define the event \[\mathcal{R}=\{|X|_{\mathcal{C}^{\beta'}}\leq R\}.\]
  Let $\hat{f}^i$ be the (random) function defined by
  \[
  \hat{f}^i( X_{i2^{-n}} +z )\coloneqq\left\{ \begin{array}{ll} \tilde{f}^i( X_{i2^{-n}} +z ) & \mbox{if } |z|\leq R2^{-\beta' n}, \\\tilde{f}^i( X_{i2^{-n}} + R2^{-\beta' n} \frac{z}{|z|}  ) &\mbox{otherwise.}
  \end{array}\right.
  \]
  In particular, $\hat{f}^i$ satisfies the following properties:
  \begin{itemize}
  \item  When restricted to the ball $ B(X_{i2^{-n}}, R 2^{-\beta' n})$, $\hat{f}^i=\tilde{f}^i$. In particular, conditionally on the event $\mathcal{R}$, $\hat{f}^i(\mathcal{D}^i_n )=\tilde{f}^i(\mathcal{D}^i_n ) $,
  \item $|\hat{f}^i|_{\mathcal{C}^\beta}\leq |f|_{\mathcal{C}^\beta}$, and
  $\|\hat{f}^i\|_{\infty}\leq R^\beta 2^{-\beta \beta' n} |f|_{\mathcal{C}^\beta}$,
  %
  \item As a random variable, $\hat{f}^i$ is measurable with respect to $X_{i2^{-n} }$.
  \end{itemize}

  Set also the translated and rescaled function $\check{f}^i(z)\coloneqq\hat{f}^i(X_{i2^{-n}}+ 2^{-\frac{n}{2}} z ) $, and $\check{X}^i:s\in [0,1]\mapsto 2^{\frac{n}{2}} (X_{(i+s)2^{-n}} -X_{i2^{-n}}) $, which is a standard planar Brownian motion started from $0$, independent from $X_{i2^{-n}}$.
These are such that
\[
\n_{\check{X}^i}(z)\check{f}^i(z)=\n_{X^i}(2^{-n/2}z+X_{i2^{-n}} )f(2^{-n/2}z+X_{i2^{-n}}).
\]
It is easily checked that translations and scaling behave with $\dashint$ as expected from the change of variable $w=rz+a$ (where $r>0$ and $a\in \mathbb{R}^2$) combined with the identity $\n_{X}(rz+a) =\n_{r^{-1}(X-a)}(z))$ : for $g$ the function $z \mapsto f(rz+a)$,
\[ \dashint \n_X f\d \lambda =  r^2\dashint \n_{r^{-1}(X-a)} g \d \lambda.\]
From the two last relations, we deduce
\[\dashint \n_{X^i} \hat{f}^i\d \lambda=
2^{-n} \dashint \n_{\check{X}^i}\check{f}^i\d \lambda.
\]
Since
  $\|\check{f}^i\|_{\infty}=\|\hat{f}^i\|_\infty\leq R^\beta 2^{-\beta \beta' n} |f|_{\mathcal{C}^\beta} $ and
  $ |\check{f}^i|_{\mathcal{C}^\beta}=2^{-\frac{\beta n}{2} } |\hat{f}^i|_{\mathcal{C}^\beta}\leq 2^{-\frac{\beta n}{2} }   |f|_{\mathcal{C}^\beta}\leq 2^{-\beta \beta' n} |f|_{\mathcal{C}^\beta}$,
  \[\|\check{f}^i\|_{\mathcal{C}_\infty^\beta}\leq (R^\beta+1) 2^{-\beta \beta' n}  |f|_{\mathcal{C}^\beta}.\]
  On the event $\mathcal{R}$, we have
  \[ \dashint \n_{X^i} \tilde{f}^i\d \lambda=\dashint \n_{X^i} \hat{f}^i\d \lambda= 2^{-n}\dashint \n_{\check{X}^i} \check{f}^i\d \lambda.\]
  Using Corollary~\ref{coro:lpbound} with $p=1$, we deduce
  \begin{align*}
  \mathbb{E}\Big[ \mathbbm{1}_{\mathcal{R}} \Big| \dashint \n_{X^i} \tilde{f}^i\d \lambda\Big| \Big]
  &=2^{-n} \mathbb{E}\left[ \mathbb{E}\left[\mathbbm{1}_{\mathcal{R}} \big|\dashint \n_{\check{X}^i} \check{f}^i\d \lambda\big| \middle| X_{i2^{-n}}\right]\right]\\
  &\leq 2^{-n} \mathbb{E}\big[ C \|\check{f}^i\|_{\mathcal{C}^\beta_\infty} \big]\\
  &\leq C(R^\beta+1) 2^{ -n-\beta\beta' n} |f|_{\mathcal{C}^\beta}.
  \end{align*}
  The key point is that we have managed not only to gain a factor $2^{-n}$ from scaling (which we will lose upon summation over the $2^n$ terms), but also an extra factor $2^{-\beta\beta'n}$ from the fact that $\tilde{f}^i$ vanishes at $X_{i2^{-n}}$.
  Thus, for all $\epsilon>0$,
  \begin{align*}
  \mathbb{P}\Big(\mathcal{R} \mbox{ and }\exists n\geq n_0: \Big| \sum_{i=0}^{2^n-1}\dashint \n_{X^i} \tilde{f}^i\d \lambda \Big|\geq \epsilon \Big)
  &\leq \frac{1}{\epsilon} \sum_{n=n_0}^\infty \sum_{i=0}^{2^n-1}
  \mathbb{E}\Big[ \mathbbm{1}_{\mathcal{R}} \Big| \dashint \n_{X^i} \tilde{f}^i\d \lambda\Big| \Big]\\
  &\leq C_{\beta,\beta', \epsilon,R } 2^{-\beta \beta' n_0} |f|_{\mathcal{C}^\beta}\\
  & \underset{n_0\to\infty}\longrightarrow 0.
  \end{align*}
  Since this holds for all $R$, we deduce that $\sum_{i=0}^{2^n-1}\dashint \n_{X^i} \tilde{f}^i\d \lambda$ almost surely converges toward $0$ as $n\to \infty$,   which concludes the proof.
\end{proof}

\subsection{Proof of Lemma~\ref{le:strato} }
In order to complete the proof of Theorem~\ref{th:green}, it only remains to prove lemma~\ref{le:strato}, which, for $\eta\in \mathcal{C}^{1+\beta}_2(\mathbb{R}^2,\mathbb{R}^2)$, identifies the limit
\[ \lim_{n\to \infty} \int_{X^{(n)}} \eta \]
with the Stratonovich integral of $\eta$ along $X$. This is classical, and follows for example from the following.
\begin{lemma}
  For a dissection $\Delta=(t_0=0<t_1<\dots< t_n=1)$, and $X:[0,1]\to \mathbb{R}^2$ a Brownian motion,
  let $X_\Delta$ be the piecewise-linear approximation of $X$ associated with $\Delta$: for $\lambda\in [0,1]$ and $t=\lambda t_i+(1-\lambda)t_{i+1}$,
  \[ X_\Delta(t)\coloneqq \lambda X_{t_i}+(1-\lambda) X_{t_{i+1}}.\]
  For $f\in \mathcal{C}^{1}(\mathbb{R}^2)$, let
  \begin{align*}
  I^1_\Delta(f)&= \sum_{[t_i,t_{i+1}]\in \Delta} f\Big(\frac{X_{t_{i+1}}+X_{t_i} }{2}\Big)(X^1(t_{i+1})-X^1(t_{i})),\\
  I^2_\Delta(f)&= \sum_{[t_i,t_{i+1}]\in \Delta} \frac{f(X_{t_{i+1}})+f(X_{t_i}) }{2}(X^1(t_{i+1})-X^1(t_{i})),\\
  I^3_\Delta(f)&= \int_{0}^1 f(X_\Delta(t)) \d X_\Delta(t).
  \end{align*}

  Then, almost surely, for all $f\in \mathcal{C}^{1+\beta}(\mathbb{R}^2)$, as the mesh $\max_i |t_{i+1}-t_i|$ of $\Delta$ goes to $0$,
  \[ I^2_\Delta(f)-I^1_\Delta(f)\to 0 \quad \mbox{and} \quad I^3_\Delta(f)-I^1_\Delta(f)\to 0.\]
\end{lemma}
\begin{proof}
  Let $\beta'\in \big(\tfrac{1}{2+\beta},\tfrac{1}{2}\big)$. Remark for $x,h\in \mathbb{R}^2$, $|f(x+h)-f(x)-\nabla_{h}f(x)|\leq |h|^{1+\beta} |\nabla f|_{\mathcal{C}^\beta}$, and thus
  \[
|f(x+h)+f(x-h)-2f(x)|\leq |f(x+h)-f(x)-\nabla_{h}f(x)|+|f(x-h)-f(x)-\nabla_{-h}f(x)|\leq 2|h|^{1+\beta} |\nabla f|_{\mathcal{C}^\beta}.
  \]
  On the almost sure event $|X|_{\mathcal{C}^{\beta'}}<\infty$, we have
  \begin{align*}
  |I^2_\Delta(f)-I^1_\Delta(f)|&=
  \Big|\sum_{[t_i,t_{i+1}]\in \Delta} \Big( \frac{f(X_{t_{i+1}})+f(X_{t_i}) }{2}-f\big( \tfrac{X_{t_{i+1}}+X_{t_i} }{2}\big) \Big)(X^1_{t_{i+1}}-X^1_{t_{i}})  \Big| \\
  &\leq \sum_{[t_i,t_{i+1}]\in \Delta}\frac{1}{2}\Big| f(X_{t_{i+1}})+f(X_{t_i})- 2 f\big( \tfrac{X_{t_{i+1}}+X_{t_i} }{2}\big)\Big| \Big|X^1_{t_{i+1}}-X^1_{t_{i}}\Big|\\
  &\leq \sum_{[t_i,t_{i+1}]\in \Delta} |\nabla f|_{\mathcal{C}^{\beta}} |X_{t_{i+1}}-X_{t_i}|^{1+1+\beta}\\
  &\leq |\nabla f|_{\mathcal{C}^{\beta}} |X|_{\mathcal{C}^{\beta'} }^{2+\beta}\sum_{[t_i,t_{i+1}]\in \Delta}  |t_{i+1}-t_i|^{\beta'(2+\beta)}\underset{|\Delta|\to 0}\longrightarrow 0.
  \end{align*}
Remark the convergence is indeed uniform over dissections, there is no need to restrict ourself e.g. to dyadic dissection as it is the case for example to ensure the almost sure convergence of $I^1_\Delta(f)$.

  The second convergence is proved in a similar way:
  \begin{align*}
  &\Big|\sum_{[t_i,t_{i+1}]\in \Delta} \Big( \int_{t_i}^{t_{i+1}} f(X_\Delta(s)) \d X_\Delta(s) - f\big( \tfrac{X_{t_{i+1}}+X_{t_i} }{2}\big) (X^1_{t_{i+1}}-X^1_{t_{i}})\Big)  \Big| \\
  &\leq \sum_{[t_i,t_{i+1}]\in \Delta}  | X^1_{t_{i+1}}-X^1_{t_{i}}| \\
  &\hspace{1cm}\int_{\frac{1}{2}}^1 \Big|f(\lambda X_{t_i}+(1-\lambda) X_{t_{i+1}} )+  f((1-\lambda)X_{t_i}+\lambda X_{t_{i+1}} )-2 f \Big( \frac{X_{t_{i+1}}+X_{t_i} }{2}\Big)\Big|\d \lambda\\
  &\leq  |\nabla f\|_{\mathcal{C}^{1+\beta}}|X|_{\mathcal{C}^{\beta'} }^{2+\beta}\sum_{[t_i,t_{i+1}]\in \Delta}  |t_{i+1}-t_i|^{\beta'(2+\beta)}\underset{|\Delta|\to 0}\longrightarrow 0.\qedhere
  \end{align*}
\end{proof}
Lemma~\ref{le:strato} now follow from the fact that, along dyadic dissections $\Delta$, $I^1_\Delta(f)$ converges toward the Stratonovich integral $\int f(X)\circ \d X$.

As we have now concluded the proof of Theorem~\ref{th:green}, we now deduce Corollary~\ref{coro:green}.
\begin{proof}[Proof of Corollary~\ref{coro:green}]
  To keep the proof short, we treat the case when $X:[0,1]\to \mathbb{R}^2$ is a Brownian loop started from $0$. For the general case, we must also take into account the winding function of the triangle between $x$, $y$, and $X_{\frac{1}{2}}$, which is straightforward.

  Again it suffices to consider the case $f\geq 0$ and we can assume  $\int_{B(0,\|X\|_\infty)} f(z)\d z>0$.
  Let $X^1$ be the restriction of $X$ to $[0,\tfrac{1}{2}]$ , $X^2$ its restriction to $[\tfrac{1}{2},1]$, and $\hat{X}_2:t\in [0,\tfrac{1}{2}]\mapsto X_{1-t}$. Then, the distribution of $X^1$ (resp. $\hat{X}^2$) admits a density with respect to the density of a standard planar Brownian motion defined on $[0,\tfrac{1}{2}]$.
  Using scale invariance, we can apply Theorem~\ref{th:green} to both $X^1$ and $\hat{X}^2$.
  We deduce that, for $i\in \{1,2\}$, for all $\beta>0$, almost surely, for all $f\in \mathcal{C}^\beta(\mathbb{R}^2)$,
  \[ \int_{\mathbb{R}^2} \cut{N}{\n_{X^i}(z)} f(z)\d z \]
  converges as $N\to \infty$, and the limits are  almost surely equal to respectively $\int_{X^1} \eta + \int_{[X_{\frac{1}{2}},0]} \eta$ and  $\int_{X^2} \eta- \int_{[X_{\frac{1}{2}},0]} \eta$, where $\eta$ is such that $\partial_1\eta_2-\partial_2\eta_1=f$.

  It only remains to prove that almost surely, for all $f\in \mathcal{C}^{\beta}(\mathbb{R}^2)$,
  \begin{equation}
\label{eq:sum}
  \dashint \n_{X^1}\d \lambda+\dashint \n_{X^2}f\d \lambda=\dashint \n_X f \d \lambda,
  \end{equation} for which we proceed as in Lemma~\ref{le:line}, introducing again a random point $P$. Going through the same arguments as in the proof of Lemma~\ref{le:line}, we see that it suffices to show that there exists $\delta>0$ such that $X$-almost surely, for $\varepsilon_1,\varepsilon_1\in \{\pm 1\}$,
  \begin{equation} \label{eq:temp1}|\mathcal{D}^{X^1}_{\varepsilon_1 N}\cap \mathcal{D}^{X^2}_{\varepsilon_2 N}|=o(N^{-1-\delta}).\end{equation}
Be careful that the probability distribution of the pair $(X^1,\hat{X}^2)$ is \emph{not} absolutely continuous with respect to that of a independent Brownian motion, as they are correlated via both their endpoints. Thus, we cannot directly rely on  Corollary~\ref{coro:joint}. Instead, we further decompose $X^1$ and $X^2$ by setting, for $i\in \{1,2,3,4\}$, $Y^i$ the restriction of $X$ to the interval $[\tfrac{i-1}{4},\tfrac{i}{4}]$, and $\hat{Y}^i$ the time-reversal.
  Then, $\mathcal{D}^{X^i}_{\pm N}\subseteq \mathcal{D}^{Y^{2i-1} }_{\pm N'}\cup \mathcal{D}^{Y^{2i}}_{\pm N'} $ where $N'=\lceil N/2\rceil$.

  Now, for any $i<j$, the distribution of $(\hat{Y}^{i},Y^{j})$ (up to the approriate recentering) is absolutely continuous with respect to that of a pair of Brownian motions (with the second one started from a random location, when $j\neq i+1$). We can thus apply  Corollary~\ref{coro:joint} to these processes and deduce indeed that~\eqref{eq:temp1} holds, which allows to apply Lemma~\ref{le:line} and conclude that~\eqref{eq:sum} holds.
\end{proof}

\section{Magnetic impurities}

In this section, we fix a function $g\in \mathcal{C}^\beta_2(\mathbb{R}^2)$. For all $\rho>0$, we define $\mathcal{P}_\rho$ a Poisson process on $\mathbb{R}^2$ with intensity $\rho g(z) \d z$, independent from $X$, and  $\Gamma:[0,T]\to \mathbb{R}$ a standard Cauchy process, independent from $X$. We write $\mathbb{E}^\mathcal{P}$
the expectation with respect to $\mathcal{P}_\rho$, $\mathbb{E}^X$ the one with respect to $X$, $\mathbb{E}^\Gamma$ the expectation with respect to $\Gamma$ and $\mathbb{E}=\mathbb{E}^X\otimes \mathbb{E}^\mathcal{P}\otimes\mathbb{E}^{\Gamma}$ the expectation on the product space (although none of the variables we consider depend on both $\mathcal{P}$ and $\Gamma$, so truly $\mathbb{E}=\mathbb{E}^X\otimes \mathbb{E}^\mathcal{P}$ or $\mathbb{E}=\mathbb{E}^X\otimes \mathbb{E}^\Gamma$, whichever is relevant).


For a function $f\in \mathcal{C}_2^\beta(\mathbb{R}^2)$, we define
\[
\xi_\rho(f)\coloneqq \frac{1}{\rho}\sum_{z\in \mathcal{P}_\rho} f(z) \n_X(z), \]
as well as
\[
\xi(f)\coloneqq \dashint \n_X  (fg) \d \lambda+ \frac{1}{2}\int_0^1 (f g)(X_t) \d \Gamma_t.
\]
Remark since $X$ and $\Gamma$ are independent, the last integral is not really a stochastic one, in the sense that, $X$-almost surely, it is the integral of a deterministic continuous function with respect to $\Gamma$.
\begin{lemma}
  \label{le:param}
  Let $f\in \mathcal{C}([0,1],\mathbb{R})$. Then, $\int_0^1 f(t) \d \Gamma_t$ is a centred Cauchy random variable with scale parameter $\int |f(t)|\d t $.
\end{lemma}
\begin{proof}
  Consider instead the case when $f=\sum_{i=1}^k f_i \mathbbm{1}_{(t_{i-1},t_{i}]}$, for $t_0=0<t_1<\dots <t_k=1$ and $f_1,\dots f_k\in \mathbb{R}$. Then,
  $\int_0^1 f(t) \d \Gamma_t = \sum_{i=1}^k f_i (\Gamma_{t_i}-\Gamma(t_{i-1})$. Since $( \Gamma_{t_i}-\Gamma(t_{i-1}) )_{i \in \{1,\dots, k\}} $ is an independent family of centered Cauchy random variables with parameters  $(t_i -t_{i-1} )_{i \in \{1,\dots, k\}} $, we deduce
  $\int_0^1 f(t) \d \Gamma_t $ is a centered Cauchy random variable with scale parameter $\sum_{i=1}^k |f_i|(t_i -t_{i-1} )$. By density this extends to any càglàd function $f$, in particular to
  $f\in \mathcal{C}([0,1],\mathbb{R})$.
\end{proof}


The main result from this section is the following
\begin{lemma}
  \label{le:magnetic}
  Let $f,g\in \mathcal{C}^\beta(\R^2)$ bounded. Assume that $g$ takes non-negative values. Let
  \[G_{\alpha,f,g}\coloneqq\sum_{k\neq 0} \int_{\mathcal{A}_k} (e^{i   \alpha k f(z)}-1) g(z)\d z.\]

  Then, $X$-almost surely, as $\alpha\to 0$,
  \begin{equation}
  G_{\alpha,f,g}\underset{\alpha \to 0}= i\alpha\dashint \n_Xfg\d \lambda  - \frac{|\alpha|}{2}   \int_0^1 |f(X_t)| g(X_t)\d t +o(\alpha).
  \label{eq:Gasymp}
  \end{equation}
\end{lemma}
Before we dive into the proof of this lemma, we first explain why this implies both Theorem~\ref{th:magne} and Corollary~\ref{coro:magne}.

\begin{proof}[Lemma~\ref{le:magnetic} implies Theorem~\ref{th:magne}]
  Since the function $\min ( |\n_X\cdot f|,1)$ is integrable against the intensity measure $\rho g \d z$ of $\mathcal{P}_\rho$, we can use Campbell's theorem, which gives
  \[
  \mathbb{E}^{\mathcal{P}}[ e^{i\theta \xi_\rho(f)} ]
  =
  \exp\Big( \sum_{k\neq 0} \int_{\mathcal{A}_k} (e^{i k\frac{\theta}{\rho} f(z)}-1) \rho g(z)\d z \Big)=\exp( \rho \ G_{\frac{\theta}{\rho},f,g}).
  \]
  Besides, by Lemma~\ref{le:param}, conditional on $X$, $\int_0^1 f(X_t) g(X_t)\d \Gamma_t$ is a centered Cauchy random variable with scale parameter $\int_0^1 |f(X_t)|g(X_t)\d t$, whilst $\dashint \n_Xfg\d \lambda$ is deterministic. It follows that
  \[\mathbb{E}^{\Gamma }[ e^{i\alpha \xi(f)} ]=e^{i\alpha  \dashint \n_Xfg\d \lambda  - \frac{|\alpha|}{2} \int_0^1 |f(X_t)|g(X_t)\d t } ,
   \]
   By taking $\alpha=\theta/\rho$, Lemma~\ref{le:magnetic} ensures that, $X$-almost surely, for all $\theta\in \mathbb{R}$,
   \[
   \mathbb{E}^{\mathcal{P}}[ e^{i\theta \xi_\rho(f)} ]
=
   \exp(\rho\ G_{\theta/\rho,f,g})\underset{\rho\to \infty}\longrightarrow \exp(i\theta   \dashint \n_Xfg\d \rho  - \frac{|\theta|}{2} \int_0^1 |f(X_t)|g(X_t)\d t   )
=\mathbb{E}^{\Gamma }[ e^{i\theta \xi(f)} ]
   ,\]
which is Theorem~\ref{th:magne}.
\end{proof}
\begin{proof}[Theorem~\ref{th:magne} implies  Corollary~\ref{coro:magne}]
Since both  $\xi_\rho(f)$ and $\xi(f)$ are linear in $f$, one can use the Cram\'er-Wold device to deduce Corollary~\ref{coro:magne} from its special case $n=1$. By Lévy's continuity theorem, this specific case is equivalent to the statement that $X$-almost surely, for all $\theta\in \mathbb{R}$,
   \[\mathbb{E}^{\mathcal{P}}[ e^{i\theta \xi_\rho(f)} ]\underset{\rho\to \infty}{\longrightarrow} \mathbb{E}^{ \Gamma}[ e^{i\theta \xi(f)} ]. \qedhere\]
\end{proof}

\begin{proof}[Proof of Lemma~\ref{le:magnetic}]
For the entirety of the proof, the realisation of $X$ is fixed, and all the constants are random.   From symmetry, we can assume $\alpha>0$.
Recall  \[G_{\alpha,f,g}\coloneqq\sum_{k\neq 0} \int_{\mathcal{A}_k} (e^{i   \alpha k f(z)}-1) g(z)\d z,\]
where the sum is absolutely convergent since $g$ is bounded, $\bigcup_{k\neq 0} \mathcal{A}_k$ is bounded, and $|e^{i  k  \alpha f(z)}-1|$ is bounded by $
2$ and vanishes on $\mathcal{A}_0$.

Consider
$G^n_{\alpha,f,g}\coloneqq\sum_{k=1}^n \int_{\mathcal{A}_k} (e^{i   \alpha k f(z)}-1) g(z)\d z$. We intervert the finite sum with the integral,
\[G^n_{\alpha,f,g}= \int_{\mathbb{R}^2}\sum_{k=0}^n \mathbbm{1}_{\mathcal{A}_k}(z) (e^{i   \alpha k f(z)}-1) g(z)\d z
\]
(the extra term for $k=0$ vanishes).

By summation by part (with $a_k\coloneqq e^{i   \alpha k f(z)}-1$, $b_k\coloneqq \mathbbm{1}_{\mathcal{D}_k}$), for any given $z$,
\begin{align*}
\sum_{k=0}^n (e^{i   \alpha k f(z)}-1)\mathbbm{1}_{\mathcal{A}_k}(z)
&=(1-e^{i\alpha (n+1)f(z)} )\mathbbm{1}_{D_{n+1}}(z)+\sum_{k=0}^n \mathbbm{1}_{\mathcal{D}_{k+1}}(z) (e^{i   \alpha (k+1) f(z)}-e^{i   \alpha k f(z)})\\
&=(1-e^{i\alpha (n+1)f(z)} )\mathbbm{1}_{D_{n+1}}(z)+\sum_{k=1}^{n+1} \mathbbm{1}_{\mathcal{D}_{k}}(z) e^{i   \alpha k f(z)} (1-e^{- i   \alpha f(z)})
\end{align*}
Inverting back the sum with the integral, we get
\[G^n_{\alpha,f,g}=
\int_{D_{n+1}} (1-e^{i\alpha (n+1)f} ) g \d \lambda+\sum_{k=1}^{n+1} \int_{\mathcal{D}_{k}}  e^{i   \alpha k f}  (1-e^{-i   \alpha  f})g \d \lambda.\]
Taking the limit $n\to \infty$, we deduce
\[\sum_{k=1}^\infty \int_{\mathcal{A}_k} (e^{i   \alpha k f}-1) g\d \lambda=
\sum_{k=1}^{\infty} \int_{\mathcal{D}_{k}}  e^{i   \alpha k f}  (1-e^{-i   \alpha  f})g \d \lambda,\]
where the sum on the left-hand side is absolutely convergent but the one on the right-hand side is convergent but may not be absolutely convergent (and in fact, it is not, when for example $f=g=1$).

With similar computation for $k<0$, we deduce
  \begin{align*}
  G_{\alpha,f,g}
  &= \sum_{k=1}^\infty \Big( \int_{\mathcal{D}_k} e^{i\alpha k f} (1-e^{-i\alpha f}) g \d \lambda+ \int_{\mathcal{D}_{-k}} e^{-i\alpha k f} (1-e^{i\alpha f}) g \d \lambda\Big)\\
  &= \sum_{k=1}^\infty (\phi_{k,\alpha}(\mathcal{D}_k )+\phi_{-k,\alpha}(\mathcal{D}_{-k})),
  \end{align*}
  where \[\phi_{k,\alpha}\coloneqq e^{i\alpha k f} (1-e^{ -i \sgn(k) \alpha f}) g.\]

  The two terms in the limit~\eqref{eq:Gasymp} come from two different parts in this last sum: the term $i\alpha \dashint \n_Xfg\d \lambda$ comes from the \emph{bulk} of the sum, that is the part with $k$ of the order of $1$. The second term comes from the \emph{tail} of the sum, or more precisely from the part of the sum when $k$ is of the order of $\alpha^{-1}$.
  We will split the sum into several parts. For $n\in \mathbb{N},\ N\in \mathbb{N}\cup \{\infty\}$ with $n<N$, set
  \[ G_{\alpha,f,g}^{n,N}=\sum_{k=n+1}^N (\phi_{k,\alpha}(\mathcal{D}_k )+\phi_{-k,\alpha}(\mathcal{D}_{-k})).\]

  For random $N_1=N_1(\alpha)$ and $N_2=N_2(\alpha)>N_1(\alpha)$ which will be set later on, we decompose $G_{\alpha,f,g}$
  into
  three parts,
  \[ G_{\alpha,f,g}= \underbrace{G_{\alpha,f,g}^{0,N_1} }_{\text{bulk}}  +\underbrace{G_{\alpha,f,g}^{N_1,N_2}}_{\text{tail}} +\underbrace{G_{\alpha,f,g}^{N_2,\infty}}_{\text{end}}.
  \]
  As $\alpha\to 0$, both $N_1$ and $\alpha N_2$ will slowly diverge toward $\infty$. In particular, $1\ll N_1\ll \alpha^{-1} \ll N_2$.
  The reason why we need to treat the \emph{end} part in a separate way is that its convergence toward $0$ is not \emph{absolute}, i.e.
  \[
  \sum_{k=N_2+1}^\infty |\phi_{k,\alpha}(\mathcal{D}_k )+\phi_{-k,\alpha}(\mathcal{D}_{-k})|
  \]
  does \textbf{not} converge toward zero as $\alpha\to 0$, and one must be careful when dealing with this term. The general term (without the absolute values) slowly oscillates between positive and negative values, and we must take advantage of sinusoidal oscillations and compensations.

We now estimate the \emph{bulk} part. For fixed $k\neq 0$, as $\alpha\to 0$,  uniformly in $z$,
  \[ \phi_{k,\alpha}(z)= \sgn(k) i \alpha f(z) g(z)+O(\alpha^2),\]
  and it follows that
  \[
  \phi_{k,\alpha}(\mathcal{D}_k)+\phi_{-k,\alpha}(\mathcal{D}_{-k})
  = i\alpha ( (fg)(\mathcal{D}_k)-(fg)(\mathcal{D}_{-k}) )+O(\alpha^2).
  \]
Let $C_k$ be such that for all $\alpha\in (0,1)$,
  \[ |   \phi_{k,\alpha}(\mathcal{D}_k)+\phi_{-k,\alpha}(\mathcal{D}_{-k}) - i\alpha ( (fg)(\mathcal{D}_k)-(fg)(\mathcal{D}_{-k}) )|\leq C_k \alpha^2, \]
  and set $N_1(\alpha)\coloneqq \min( \lfloor \alpha^{-\frac{1}{3}}\rfloor , \sup \{ N: \forall k\leq N, C_k\leq \alpha^{-\frac{1}{3}} \}) $.

  Then,
  \[ \Big| G_{\alpha,f,g}^{0,N}- i\alpha \sum_{k=1}^{N_1}   ( (fg)(\mathcal{D}_k)-(fg)(\mathcal{D}_{-k}) )\Big| \leq \sum_{k=1}^{N_1} C_k \alpha^2\leq \alpha^{\frac{4}{3}}=o(\alpha).\]
  Besides, $X$-almost surely, $N_1\underset{\alpha \to 0}\longrightarrow +\infty$, and  Theorem~\ref{th:green} implies that
  \[ \sum_{k=1}^{N_1}   ( (fg)(\mathcal{D}_k)-(fg)(\mathcal{D}_{-k}) )\underset{\alpha \to 0}\longrightarrow \dashint \n_Xfg\d \lambda.\]
  Therefore,
  \begin{equation}
  \label{eq:bulk}
  G_{\alpha,f,g}^{0,N}=i\alpha\dashint \n_Xfg\d \lambda +o(\alpha).
  \end{equation}

  We now look at the \emph{tail} part of $G_{\alpha,f,g}$.
  Let $\delta>0$ and $C$ (random) be such that for all $k\neq 0$ and $\phi\in \mathcal{C}^\beta_2(\mathbb{R}^2)$,
  \[ \Big| \phi(\mathcal{D}_k)-\frac{1}{2\pi |k|}\int_0^1 \phi(X_u) \d u\Big|\leq C \|\phi\|_{\mathcal{C}_\infty^\beta} k^{-1-\delta}.\]
  Recall that the existence of such a couple $(\delta,C)$ is provided by Lemma~\ref{le:bound}.
  Let $N_2=N_2(\alpha)$ be any integer-valued function such that $\alpha N_2 \underset{\alpha \to 0}\longrightarrow +\infty $ and $ \alpha N_2^{1-\delta} \underset{\alpha \to 0}\longrightarrow 0$.

  For all $\phi,\psi\in \mathcal{C}^\beta_2(\mathbb{R}^2)$,
  $|\phi\psi|_{\mathcal{C}^\beta}\leq |\phi|_{\mathcal{C}^\beta}\|\psi\|_\infty+\|\phi\|_\infty |\psi|_{\mathcal{C}^\beta}$. We deduce that for all $k$ and $\alpha$,
  \begin{align*}
  \|\phi_{k,\alpha}\|_{\infty}&\leq \|e^{i\alpha k f}\|_\infty\|1-e^{i\alpha f}\|_{\infty}\|g\|_{\infty}\leq \alpha \|f\|_\infty \|g\|_{\infty}, \\
  |\phi_{k,\alpha}|_{\mathcal{C}^\beta}
  &\leq |e^{i\alpha k f}|_{\mathcal{C}^\beta} \|1-e^{i\alpha f}\|_{\infty}\|g\|_{\infty}+\|e^{i\alpha k f}\|_\infty |1-e^{i\alpha f}|_{\mathcal{C}^\beta}\|g\|_{\infty}+\|e^{i\alpha k f}\|_\infty\|1-e^{i\alpha f}\|_{\infty}|g|_{\mathcal{C}^\beta}\\
  &\leq k \alpha^2 |f|_{\mathcal{C}^\beta} \|f\|_\infty \|g\|_{\infty}+ \alpha |f|_{\mathcal{C}^\beta}\|g\|_{\infty}+\alpha \|f\|_\infty |g|_{\mathcal{C}^\beta},
  \end{align*}
  so that
  \[ \| \phi_{k,\alpha}\|_{\mathcal{C}^\beta_\infty}\leq \alpha(1+k\alpha )(1+\|f\|_{\mathcal{C}^\beta_\infty})\|f\|_{\mathcal{C}^\beta_\infty}\|g\|_{\mathcal{C}^\beta_\infty}.\]

  We deduce that, for all $k> 0$,
  \begin{multline*}
  \Big| \phi_{k,\alpha}(\mathcal{D}_k)+\phi_{-k,\alpha}(\mathcal{D}_{-k}) - \frac{1}{2\pi k} \int_0^1 (\phi_{k,\alpha}(X_u)+\phi_{-k,\alpha}(X_u)) \d u \Big|
  \\\leq 2 C(1+\|f\|_{\mathcal{C}^\beta_\infty})\|f\|_{\mathcal{C}^\beta_\infty}\|g\|_{\mathcal{C}^\beta_\infty}   \alpha(1+k\alpha ) k^{-1-\delta}.
  \end{multline*}
  Thus, there exist constants $C'=C'(f,g),\ C''=C''(f,g)$ such that for all $N_2\geq N_1$,
  \begin{align*}
  \Big|G_{\alpha,f,g}^{N_1,N_2}- &\frac{1}{2\pi }\sum_{k=N_1+1}^{N_2}  \frac{1}{k} \int_0^1 (\phi_{k,\alpha}(X_u)+\phi_{-k,\alpha}(X_u)) \d u\Big| \\
  &\leq C' \sum_{k=N_1+1}^{N_2} \alpha(1+k\alpha ) k^{-1-\delta}\leq C'' \alpha ( N_1^{-\delta}+ \alpha N_2^{1-\delta})=o(\alpha).
  \end{align*}
  Set  \[\psi_{k,\alpha}\coloneqq e^{i\alpha k f} \sgn(k) i\alpha f  g.\]
  Then, for $\alpha\leq \|f\|_\infty$,
  \begin{align*}
  \Big|\sum_{k=N_1+1}^{N_2} \frac{ \phi_{k,\alpha}- \psi_{k,\alpha} }{k}\Big|
  &=|g| \Big|\sum_{k=N_1+1}^{N_2}\frac{1}{k}e^{i\alpha k f} (1-e^{-\sgn(k)i \alpha f }-\sgn(k)i\alpha f ) \Big|\\
  &\leq |g| \sum_{k=N_1+1}^{N_2} \frac{1}{k} \frac{\alpha^2 f^2}{2}\leq C_{f,g} |\log(\alpha)| \alpha^2=o(\alpha).
  \end{align*}
  It follows that
  \begin{align*}
  G_{\alpha,f,g}^{N_1,N_2}&= \frac{1}{2\pi }\sum_{k=N_1+1}^{N_2}  \frac{1}{k} \int_0^1 (\psi_{k,\alpha}(X_u)+\psi_{-k,\alpha}(X_u)) \d u+o(\alpha)\\
  &= -\frac{\alpha}{\pi }\sum_{k=N_1+1}^{N_2} \int_0^1  f(X_u)g(X_u) \frac{\sin( k\alpha f(X_u))}{k} \d u+o(\alpha)\\
  &= -\frac{\alpha}{\pi }\sum_{k=1}^{N_2} \int_0^1  f(X_u)g(X_u) \frac{\sin( k\alpha f(X_u))}{k} \d u+o(\alpha).
  \end{align*}
  The last line follows from the fact that
  \[ \Big|\frac{\alpha}{\pi }\sum_{k=1}^{N_1} \int_0^1  f(X_u)g(X_u) \frac{\sin( k\alpha f(X_u))}{k} \d u \Big| \leq \|f\|_\infty^2\|g\|_\infty \alpha^2 N_1=o(\alpha). \]

  For $s\leq 0$, let \[\Phi(s)=\left\{ \begin{array}{ll} \int_0^1 f(X_u)g(X_u)\frac{\sin(sf(X_u))}{s} \d u & \mbox{for } s\neq 0 \\
  \int_0^1 f(X_u)^2g(X_u)\d u & \mbox{for } s= 0,
  \end{array} \right. \]
  so that $\Phi$ is continuous on $[0,\infty)$ and
  \begin{equation}
  \label{eq:equiv1}
  G_{\alpha,f,g}^{N_1,N_2}=-\frac{\alpha^2}{\pi }\sum_{k=1}^{N_2} \Phi( \alpha k)+o(\alpha).
  \end{equation}

  For all $R>0$,
  \[ \Big|\alpha \sum_{k=1}^{ \lfloor R \alpha^{-1}\rfloor  } \Phi(\alpha k)-\int_0^R \Phi(s)\d s \Big|\leq \alpha \|f\|_\infty^2 \|g\|_\infty+ \omega_{\Phi,[0,R]}(\alpha),\]
  where $\omega_{\Phi,[0,R]}(\alpha)\coloneqq \sup_{s,t\in [0,R]} |\Phi(s)-\Phi(t)|$ is the continuity modulus of $\Phi$.

  Since $\omega_{\Phi,[0,R]}(\alpha)\underset{\alpha\to 0}\longrightarrow 0$ for all $R>0$, there exists a function $R_\alpha$ such that $R_\alpha \to \infty$ as $\alpha \to 0$ and
  $ \omega_{\Phi,[0,R_\alpha]}(\alpha)\underset{\alpha\to 0}\longrightarrow 0$. We fix such a function, and set $N_2(\alpha)\coloneqq \alpha^{-\frac{2}{2-\delta}} \wedge ( \alpha^{-1} R_\alpha)$. Remark then that $\alpha N_2 \underset{\alpha\to 0}\longrightarrow +\infty$ and $\alpha N_2^{1-\delta} \underset{\alpha\to 0}\longrightarrow 0$, as previously required.

  We obtain
  \begin{equation}
  \label{eq:equiv2}
  \Big|\alpha \sum_{k=1}^{ N_2 } \Phi(\alpha k)-\int_0^{\alpha^{-1}N_2}  \Phi(s)\d s \Big|=o(1).
  \end{equation}

  To estimate this last integral, there is two things we must be careful about. First, because of the $\sinc$ function in the definition of $\Phi$, the function $\Phi$ is not integrable on $[0,+\infty)$, so we cannot naively replace the bound $\alpha^{-1} N_2$ with its limit. Secondly, when manipulating the integral, we must be extra careful at the vicinity of $f(X_u)=0$.

  Recall that for $x\neq 0$, $\lim_{C\to \infty} \int_0^C \frac{\sin(sx)}{s}\d s=\sgn(x)\frac{\pi}{2}$.
  Performing an integration by part, we deduce that for all $x$ and $C> 0$,
  \begin{align*}
  \Big| \int_0^C \frac{\sin(sx)}{s} \d s-\sgn(x)\frac{\pi}{2}\Big|
  &=\Big|\lim_{C'\to \infty} \int_C^{C'} \frac{\sin(sx)}{s} \d s\Big|\\
  &=\Big| \frac{\cos(C x)}{Cx}- \lim_{C'\to \infty} \int_C^{C'}  \frac{\cos(sx)}{s^2x}\d s\Big| \\
  &\leq \frac{2}{C|x|}.
  \end{align*}
  It follows that
  \begin{align}
  \Big|\int_0^{\alpha^{-1}N_2} \Phi(s)\d s &-\frac{\pi}{2} \int_0^1 |f(X_u)|g(X_u)\d u\Big|\nonumber\\
  &=\Big|  \int_0^1 f(X_u)g(X_u) \Big(  \int_0^{\alpha^{-1}N_2} \frac{\sin(sf(X_u))}{s} \d s-\sgn(f(X_u))\frac{\pi}{2}     \Big)   \d u\Big|\nonumber\\
  &\leq \int_0^1 |f(X_u)|g(X_u) \frac{2}{{\alpha^{-1}N_2}|f(X_u)|} \d u\nonumber\\
  &=O(\alpha N_2^{-1})=o(1).\label{eq:equiv3}
  \end{align}

  Combining~\eqref{eq:equiv1},~\eqref{eq:equiv2} and~\eqref{eq:equiv3}, we obtain
  \begin{equation}
  \label{eq:tail}
  G^{N_1,N_2}_{\alpha,f,g}=-\frac{\alpha}{2}\int_0^1 |f(X_u)|g(X_u)\d u+o(\alpha).
  \end{equation}


  We finally look at the \emph{end} part of $G_{\alpha,f,g}$. Since the $\mathcal{C}^\beta$ norm of $\phi_{k,\alpha}$ becomes arbitrarily large as $k$ goes to infinity, one cannot directly rely on Lemma~\ref{le:bound}.
  %
  For a positive integer $j$, we decompose $G^{j^2N_2,(j+1)^2N_2}_{\alpha,f,g}$ into
  \begin{align*}
  G^{j^2N_2,(j+1)^2N_2}_{\alpha,f,g}&= \underbrace{\sum_{k=j^2N_2+1 }^{(j+1)^2N_2} (\phi_{k,\alpha}(\mathcal{D}_{(j+1)^2N_2} )-\phi_{-k,\alpha}(\mathcal{D}_{-(j+1)^2N_2}))}_{\eqqcolon H^{j}_{\alpha,f,g} } \\
  &\hspace{1cm}+  \underbrace{\sum_{k=j^2N_2+1 }^{(j+1)^2N_2} ( \phi_{k,\alpha}(\mathcal{D}_k)- \phi_{k,\alpha}(\mathcal{D}_{(j+1)^2N_2} )- \phi_{k,\alpha}(\mathcal{D}_{-k})+\phi_{-k,\alpha}(\mathcal{D}_{-(j+1)^2N_2})}_{\eqqcolon K^{j}_{\alpha,f,g}} .
  \end{align*}
The idea here is that, in the considered range for $k$, $\mathcal{D}_k$ does not vary too much (which allows to control the second term $K^j_{\alpha,f,g}$), whilst $\phi_{k,\alpha}$ varies sufficiently, along a circle centered at $0$, to allow for cancellations to happen  (which allows to control the first term $H^j_{\alpha,f,g}$).

  Indeed, we have
  \begin{align*}
  \Big| \sum_{k=j^2N_2+1 }^{(j+1)^2N_2 } \phi_{k,\alpha}(\mathcal{D}_{(j+1)^2N_2} )\Big|
  &=\Big|\int_{\mathcal{D}_{(j+1)^2N_2}} \sum_{k=j^2N_2+1 }^{(j+1)^2N_2} e^{-i\alpha k f(z) }(1-e^{-i\alpha f(z)})g(z) \d z\Big| \\
  &=\Big| \int_{\mathcal{D}_{(j+1)^2N_2} } e^{-i\alpha (j^2N_2+1) f(z) } (1-e^{-i\alpha ((j+1)^2N_2-j^2N_2) f(z) }  )g(z) \d z\Big|\\
  &\leq \int_{\mathcal{D}_{(j+1)^2N_2} }  2 |g(z)| \d z\\
  &\leq 2 \|g\|_\infty D_{(j+1)^2N_2}.
  \end{align*}

  Using again Lemma~\ref{le:bound} with $f=1$, we deduce that almost surely, there exists $C$ such that for all $N$, $D_{N}\leq  \frac{C}{N}$.
  It follows that
  \[|H^{j}_{\alpha,f,g}|\leq \frac{4C \|g\|_\infty}{ (j+1)^2N_2 }, \]
  which yields
  \[
  \Big| \sum_{j=1}^\infty H^{j}_{\alpha,f,g} \Big|\leq \frac{ 4C \|g\|_\infty}{N_2}\sum_{j=2}^\infty \frac{1}{j^2}=o(\alpha).
  \]

  As for $K^{j}_{\alpha,f,g}$, using the fact that the sequences $(\mathcal{D}_k)_{k\geq 1}$ and $(\mathcal{D}_{-k})_{k\geq 1}$ are nested (that is, $\mathcal{D}_{k+1}\subseteq \mathcal{D}_k$ and $\mathcal{D}_{-k-1}\subseteq \mathcal{D}_{-k}$ for all $k>0$), we have
  \begin{align*}
  K^{j}_{\alpha,f,g}
  &=  \sum_{k=j^2N_2+1}^{(j+1)^2N_2} |\phi_{x,\alpha}| (D_k-D_{(j+1)^2N_2}+ D_{-k}-D_{-(j+1)^2N_2} )\\
  &\leq \sum_{k=j^2N_2+1}^{(j+1)^2N_2}  \alpha \|f\|_\infty \|g\|_\infty (D_k-D_{(j+1)^2N_2} + D_{-k}-D_{-(j+1)^2N_2}).
  \end{align*}
  Let $C,\delta>0$ such that for all $N\neq 0$,
  \[ \big|D_N-\frac{1}{2\pi |N|}\big|\leq C N^{-1-\delta}.\]
  Then, for all $k\in \{j^2N_2+1,\dots, (j+1)^2N_2\}$,
  \[0\leq  D_k-D_{(j+1)^2N_2}\leq \frac{1}{2\pi k}-\frac{1}{2\pi (j+1)^2N_2} +2C k^{-1-\delta}\leq C'\big( \frac{1}{j^3 N_2^2}+ (j^2N_2)^{-1-\delta}\big).\]
  We deduce
  \[|K^{j}_{\alpha,f,g}|\leq C''  \|f\|_\infty \|g\|_\infty N_2^{-1} j^{-2},\]
  and it follows that
  \[ \sum_{j=1}^\infty |K^{j}_{\alpha,f,g}|=o(\alpha).\]

  Finally, we have
  \begin{equation} \label{eq:end} |G^{N_2,\infty}_{\alpha,f,g}|\leq \sum_{j=1}^{\infty} |G^{j^2 N_2,(j+1)^2N_2}_{\alpha,f,g}|\leq \sum_{j=1}^{\infty} |K^{j}_{\alpha,f,g}|+\sum_{j=1}^{\infty} |H^{j}_{\alpha,f,g}|=o(\alpha).\end{equation}
  We conclude the proof by putting together~\eqref{eq:bulk},~\eqref{eq:tail} and~\eqref{eq:end}.
\end{proof}

\section{Funding}
I am pleased to acknowledge support from the ERC Advanced Grant 740900 (LogCorRM), and later from the EPSRC grant EP/W006227/1 .

\bibliographystyle{plain}
\bibliography{green_25_11_26.bbl}

\end{document}